\documentclass[a4paper]{amsart}

\usepackage[english]{babel}
\usepackage[utf8]{inputenc}
\usepackage[T1]{fontenc}
\usepackage{kpfonts}
\usepackage[margin=1.4in]{geometry}
\usepackage{amssymb, amsmath, amsfonts, mathrsfs, amscd, amsthm, amsbsy}
\usepackage{url}
\usepackage{tikz}
\usepackage{stmaryrd}
\usepackage{accents}
\usetikzlibrary{matrix, arrows, positioning}

\theoremstyle{plain}
\newtheorem{theorem}{Theorem}[section]
\newtheorem{lemma}[theorem]{Lemma}
\newtheorem{proposition}[theorem]{Proposition}
\newtheorem{corollary}[theorem]{Corollary}

\theoremstyle{definition}
\newtheorem{definition}{Definition}[section]
\newtheorem{tam}{Tame Assumption}[section]

\theoremstyle{remark}
\newtheorem{remark}{Remark}
\newtheorem*{notation}{Notation}
\numberwithin{equation}{section}

\newcommand{\Z}{\mathbb{Z}}

\newcommand{\C}{\mathbb{C}}
\newcommand{\N}{\mathbb{N}}

\DeclareMathOperator{\Ker}{Ker}
\DeclareMathOperator{\Spec}{Spec}
\DeclareMathOperator{\Proj}{Proj}

\DeclareMathOperator{\Hom}{Hom}

\DeclareMathOperator{\Hhom}{\mathit{Hom}}

\DeclareMathOperator{\Div}{Div}

\DeclareMathOperator{\Aut}{Aut}

\DeclareMathOperator{\Sym}{Sym}
\DeclareMathOperator{\Br}{Br}
\DeclareMathOperator{\val}{val}

\DeclareMathOperator{\Res}{Res}
\DeclareMathOperator{\Sch}{Sch}
\DeclareMathOperator{\SchS}{Sch_{S}}
\DeclareMathOperator{\Diff}{Diff}
\DeclareMathOperator{\Ri}{R}
\DeclareMathOperator{\Pic}{Pic}
\DeclareMathOperator{\Hi}{H}
\newcommand{\Card}[1]{\mathrm{Card}\left(#1\right)} 

\newcommand{\id}{\mathrm{id}}
\newcommand{\bA}{\mathbb{A}}
\newcommand{\fB}{\mathfrak{B}}

\newcommand{\cD}{\mathcal{D}}%

\newcommand{\sE}{\mathscr{E}}%

\newcommand{\bG}{\mathbb{G}}%
\newcommand{\fH}{\mathfrak{H}}
\newcommand{\cH}{\mathcal{H}}
\newcommand{\cHh}{\smash{\overline{\mathcal{H}}}}
\newcommand{\fHh}{\smash{\overline{\mathfrak{H}}}}%
\newcommand{\bbk}{\Bbbk}
\newcommand{\sL}{\mathscr{L}}%
\newcommand{\fM}{\mathfrak{M}}
\newcommand{\fMm}{\smash{\overline{\mathfrak{M}}}}%

\newcommand{\sO}{\mathscr{O}}%
\newcommand{\bP}{\mathbb{P}}
\newcommand{\sP}{\mathscr{P}}
\newcommand{\fP}{\mathfrak{P}}%
\newcommand{\fZ}{\mathfrak{Z}}%
\newcommand{\vk}{\smash{\vec{\boldsymbol{k}}}}%

\newcommand{\mr}[1]{\mathrm{#1}}
\newcommand{\mc}[1]{\mathcal{#1}}
\newcommand{\mf}[1]{\mathfrak{#1}}
\newcommand{\mbf}[1]{\mathbf{#1}}
\newcommand{\mbb}[1]{\mathbb{#1}}
\newcommand{\ms}[1]{\mathscr{#1}}
\newcommand{\bs}[1]{\boldsymbol{#1}}

\begin{document}

\title{Comments on the ELSV compactification of Hurwitz stacks}
\author{Bashar Dudin}

\address{Universit\'e du Maine, Département de Mathématiques, Avenue Olivier Messaien, BP 535, 72085 Le Mans cedex 9, France.}
\email{bashar.dudin@univ-lemans.fr}
\date{\today}

\maketitle
\begin{abstract}
  We revisit Ekedahl, Lando, Shapiro and Vainshtein's compactification
  of the stack of simply ramified covers of the projective line except for a fixed ramification profile above infinity. In particular we draw a connection with the Harris and Mumford stack of admissible covers showing that the boundary of the ELSV compactification appears as a contraction of the boundary of the stack of admissible covers. This lays the needed foundations for a combinatoral interpretation of boundary points of the ELSV compactification. 
\end{abstract}

\setcounter{tocdepth}{1}
\tableofcontents

\section{Introduction}

Fix integers $(g, n)\in \N\times \N^*$ distinct from $(0,1)$ and $(0,2)$. A genus $g$ cover of profile $\vk = (k_1, \ldots, k_n) \in (\N^*)^n$ is a ramified cover $\phi : X \rightarrow Y$ from a genus $g$ smooth $n$-marked curve $(X, \underline{p}= p_1, \ldots, p_n)$ on a genus $0$ $1$-marked smooth curve $(Y, q_\infty)$, simply branched away from $q_\infty$ and which sends $p_i$ on $q_\infty$ (and only the $p_i$s) with multiplicity $k_i$. The Hurwitz number $\mu_{g, \vk}$ is the sum over $g$ covers $\phi$ of profile $\vk$ having $r$ fixed branch points each weighted by $1/\Aut_{0}(\phi)$.  The automorphism group $\Aut_{0}(\phi)$ is the group of automorphisms of the domain curve that fix the marked points and commute to $\phi$. In case the base field is $\C$ the celebrated ELSV formula \cite[1.1]{ELSV}  
\begin{equation}
  \label{eqELSV}
  \tag{1}
  \mu_{g, \vk}= r!\prod_{i=1}^n \frac{k_i^{k_i}}{k_i!} \int_{\fMm_{g, n}} \frac{c\left(\mbb{E}_{g, n}^{\vee}\right)}{(1-k_1\psi_1)\cdots (1-k_n\psi_n)}
\end{equation}
relates $\mu_{g, \vk}$ to an intersection number on the stack of stable genus $g$ $n$-marked curves $\fMm_{g, n}$ (\cite{DMMg}). It was announced in \cite{ELSV0} with a sign error, corrected and proved in \cite{BranchFP} in the case $\vk=(1,\ldots,1)$ and finally proved in the general case in \cite{ELSV}. Techniques in \cite{BranchFP} and later on \cite{GrabVak} depend on a fine understanding of Gromov-Witten theory of $\bP^1$, those in \cite{ELSV} build up a geometry reflecting the right-hand side of \ref{eqELSV}. This geometry is the main interest of this paper. It involves constructing compactifications of stacks of covers of $\bP^1$ that have a projective cone structure over $\fMm_{g, n}$. The question we're interested in is how are these ``ELSV'' compactifications related to more standard compactifications of covers of $\bP^1$ such as \cite{HurHM}, \cite{HurMochi}, \cite{HurACV} or \cite{HurBR}. The answer to this question is the content of theorem \ref{thm:finale}. In order to give a smooth functorial treatment of the subject we had to give a satisfactory algebro-geometric understanding of the cone of twisted polar parts \cite[3.4]{ELSV}. In sections \ref{subsec:Toroidal} and \ref{subsec:polarcurves} we give two constructions of this cone, they are isomorphic if the characteristic of the base field is big enough. The present work makes clear the fact that the ELSV formula is still valid as long as the characteristic of the base field doesn't divide any of the $k_i$s and the number of branch points away from $q_\infty$. Lastly we mention that theorem \ref{thm:finale} is the starting point of the combinatorial description of boundary points of the ELSV compactification announced in \cite{BashCRAS}. It is the subject of a forthcoming paper by the author.  

Throughout the paper $\bbk$ is an algebraically closed field of characteristic $\bs{p} \geq 0$ and schemes are algebraic schemes over $\bbk$. We denote $\cH_{g, \vk}$ the stack of $g$ covers of profile $\vk$. An automorphism of a point $\phi : X \rightarrow Y$ in $\cH_{g, \vk}$ is a couple of isomorphisms $(\alpha, \beta)$ of the domain and target, fixing the marked points and such that $\beta\phi = \phi\alpha$, details can be found in \cite{HurHM}, \cite{EmsalemHur} and \cite{HurBR}. More generally, following \cite{HurBR}, if given a set $\{\bs{\epsilon}_j\}_{j =1}^{\bar{r}}$ of partition of $d = \sum_{i=1}k_i$ one can define the stack $\cH_{g, \vk}^{\bs{\epsilon}}$ of $g$ covers that have profile $\vk$ above the only marked point of the target and ramification indices given by the partitions $\bs{\epsilon}_j$ over each branch point. Denote $\cH_{g, \vk}^{\bullet}$ the finite disjoint union of all such $\cH_{g, \vk}^{\bs{\epsilon}}$. 

The starting point of the ELSV compactification comes from the following understanding of the moduli stack of $g$ covers of profile $\vk$: denote $\mbf{P}_{g, \vk}$ the space of couples $((X, \underline{p}), \underline{\varrho})$ where $(X, \underline{p})$ is a smooth $n$-marked genus $g$ curve and $\underline{\varrho} = (\varrho_1, \ldots, \varrho_n)$ is a global section of $\sO_X(\sum_ik_ip_i)/\sO_X$ called a polar part. Locally at $p_i$ we ask for $\varrho_i$ to be of the form 
\begin{displaymath}
  \varrho_i = \frac{a_{k_i}}{t_i^{k_i}} + \frac{a_{k_i-1}}{t_i^{k_i-1}} + \cdots + \frac{a_1}{t_i} \quad \textrm{where $a_{k_i}\neq 0$}. 
\end{displaymath}
We have a natural $\bG_m$-action on the fibers of $\mbf{P}_{g, \vk}$ over $\fM_{g, n}$. The residue theorem says the zero locus of the residue linear form $\kappa \mapsto \sum_i\Res_{p_i}(\kappa\varrho_i)$ for $\kappa \in \Hi^0(X, \Omega_X)$ is precisely the set of meromorphic functions (up to the addition of a constant) whose behavior in the neighborhood of $p_i$ is given by $\varrho_i$. The quotient of this locus by the $\bG_m$-action coming from $\mbf{P}_{g, \vk}$ is $\cH_{g, \vk}$. The ELSV approach is to extend $\mbf{P}_{g, \vk}$ into a cone over the proper stack of stable curves $\fMm_{g, n}$ for which the residue theorem still makes sense. Notice first that one can still make sense out of polar parts at the marked points of any prestable genus $g$ $n$-marked curve (see \cite{BehManin}), in particular this is as well the case for stable curves. In \cite{ELSV} the authors construct the desired extension of $\mbf{P}_{g, \vk}$ locally over $\fMm_{g, n}$. Take a stable $n$-marked curve $(X, \underline{p})$. If given a local equation $t_i$ at the neighborhood of $p_i$ a polar part at $p_i$ is given by a $k_i$-tuple $(a_{k_i}, a_{k_i-1}, \ldots, a_1)$ with $a_{k_i}\neq 0$. In loc. cit. the authors choose to describe a polar part using coordinates $(u_i, c_{1}, \ldots, c_{k_i-1})$ with $u_i \neq 0$ corresponding to the polar part 
\begin{displaymath}
  \left(\frac{u_i}{t_i}\right)^{k_i} + c_{1}\left(\frac{u_i}{t_i}\right)^{k_i-1} + \cdots + c_{k_i-1}\left(\frac{u_i}{t_i}\right).
\end{displaymath}
For this to make sense one has to mod out the coordinates by the action of $\bs{\mu}_{k_i}$ having weights $(1, 1,2,\ldots, k_i)$. Extending $\mbf{P}_{g, \vk}$ is to allow $u_i=0$. The resulting extension of $\mbf{P}_{g, \vk}$ to $\fMm_{g, n}$ which shall still be written $\mbf{P}_{g, \vk}$ is locally over $\fMm_{g, n}$ isomorphic to 
\begin{displaymath}
\prod_{i=1}^n [\bA^{k_i}/\bs{\mu}_{k_i}].  
\end{displaymath}
This is the model we keep in mind through sections \ref{subsec:Toroidal} and \ref{subsec:polarcurves} while giving more natural and functorial constructions of $\mbf{P}_{g, \vk}$. 

\subsection*{Conventions on stacks}

We use the conventions of \cite{Kainc}. In particular, an algebraic stack is an algebraic stack over $\bbk$ in the sense of \cite{Laumon} that is quasi-separated and locally of finite type on $\bbk$. A Deligne--Mumford stack is an algebraic stack in the sense of definition \cite[4.6]{DMMg}, i.e. an algebraic stack whose diagonal is unramified. By a morphism of stacks we mean a $1$-morphism. A Deligne--Mumford stack comes with the Grothendieck topology of the small \'etale site. An algebraic stack that is not Deligne--Mumford has the faithfully flat of finite presentation topology. The action of a group scheme on a stack and the resulting quotient is to be understood in the sense of \cite{MRomagnyG}. An $\bA^1$-action on a stack in the sense of \cite{Kainc} is called an $\bA^1$-structure in order to avoid confusion with the notion of a group action.

\subsection*{Acknowledgements}

The present work is part of the author thesis and he would like to thank his advisor J.~Bertin for his guidance and support.

\section{Toroidal construction of twisted polar parts}
\label{subsec:Toroidal}

We give an alternative description of the stack of twisted polar parts $\bs{P}_{g, \vk}$ making use of toroidal techniques due to \cite{KKMS}. This approach is of local nature on the base. We start by focusing on the $1$-marked case.

\subsection{The group of formal diffeomorphisms of a marked point}
Let $(\pi: \mc{X} \rightarrow S, D)$ be a prestable $1$-marked curve over an affine local scheme $S=\Spec(\bs{A})$. Let $p$ be a point in the support of $D$ having image $s=\pi(p)$ and $t$ a local equation of $D$ at $p$. Since $\pi$ is smooth at $p$ the formal completion of $\sO_p$ relative to the maximal ideal is isomorphic to $\bs{A}\llbracket t\rrbracket$. The ring $\mc{P}$ of Laurent tails at $p$ can be written  
\begin{displaymath}
  \sP = \bs{A}\llbracket t \rrbracket[t^{-1}]/\bs{A}\llbracket t \rrbracket.
\end{displaymath}
An element $\varrho \in \sP$ can be uniquely written as
\begin{displaymath}
  \varrho = \sum_{0 < \ell \leq k} \frac{a_\ell}{t^{\ell}}  \quad \text{with $k >0$ and $a_k \neq 0$}
 \end{displaymath}
The integer $k$ is the order of $\varrho$. The set of Laurent tails at $p$ of order $k$ is written $\sP_{k}$, it corresponds to the set of polar parts at $p$ of order $k$.

Define $\Diff_{\bs{A}}$ as the set of $\bs{A}$-automorphisms of $\bs{A}\llbracket t \rrbracket$ fixing the section $t=0$. It is composed of elements $f(t) = \alpha_0t^1 + \alpha_2t^2 + \cdots$ such that $\alpha_0 \in \bs{A}^*$ and has a group structure given by $f(t)h(t) = h(f(t))$. One has a left action of $\Diff_{\bs{A}}$ on $\sP$ defined by $f(t) \cdot \varrho(t)  = \varrho(f(t))_<$ where the right-hand side denotes the Laurent tail of $\varrho(f(t))$. It is clear that $\varrho(t)$ is invariant under the action of an element of the form $f(t) = t + \alpha_kt^{k+1} + \alpha_{k+1}t^{k+2} + \cdots$. Such elements define a normal subgroup $\Diff_{\bs{A}}(k)$ of $\Diff_{\bs{A}}$. The elements of the quotient group $\Diff_{\bs{A}, k} = \Diff_{\bs{A}}/\Diff_{\bs{A}}(k)$ are of the form $f(t) = \alpha_0t^1 + \alpha_1t^2 + \cdots + \alpha_{k-1}t^k$. 
\begin{lemma}
  $\Diff_{\bs{A}, k}$ acts transitively on the set of Laurent tails of order $k$ with isotropy group $\bs{\mu}_{k}$. 
\end{lemma}
\begin{proof}
  Transitivity results from the fact any such Laurent tail can be written $\varrho(t) = t^{-k}\varphi(t)_<$ for a truncated invertible formal series $\varphi(t)_<$. The stabilizer of $t^{-k}$ is clearly $\bs{\mu}_{\vk}$. 
\end{proof}
To look for affine embeddings of $\sP_k$ we study affine equivariant embeddings of $\Diff_{\bs{A}, k}$.  

\subsection{The group structure of $\Diff_{k}$}
Recall the algebraic group $\bG_{m, \bs{A}}$ is the spectrum of the graded $\bs{A}$-algebra \begin{displaymath}
\bs{A}[M] = \bigoplus_{\gamma \in \Z} \bs{A}\chi^\gamma. 
\end{displaymath} 
The set of variables $M$ is  the group of characters of $\bG_{m, \bs{A}}$. Denote $N$ the group of $1$-parameter subgroups of $\bG_{m, \bs{A}}$. 

In the sequel we drop the index $\bs{A}$ relative to the base scheme we work on. We thus write $\bG_m$ instead of $\bG_{m, \bs{A}}$ and $\Diff_{k}$ instead of $\Diff_{\bs{A}, k}$. 

The group $\Diff(2)$ defines a normal subgroup of $\Diff_{k}$. The corresponding quotient group is equal to $\Diff_{1}$ which is the automorphism group of $\widehat{\sO}_p/\widehat{\sO}_p(-p)$ and is thus equal to $\bG_m$. This gives an extension of groups
\begin{equation}
\label{suite:unipotent}
  \begin{tikzpicture}[>=latex, baseline=(current  bounding  box.center)]
    \matrix (m) 
    [matrix of math nodes, row sep=1.5em, column sep=2.51em, text
    height=1ex, text depth=0.25ex]  
    { 1 & \mc{U}_k & \Diff_{k}  & \bG_m & 1. \\}; 
    \path[->,font=\scriptsize]  
    (m-1-1) edge  (m-1-2)
    (m-1-2) edge (m-1-3)
    (m-1-3) edge (m-1-4)
    (m-1-4) edge (m-1-5);
  \end{tikzpicture}      
\end{equation}
Given a local coordinate $t$ at $p$ an element $f(t)  \in \mc{U}_k$ is of the form $f(t) = t + \alpha_1t^2 + \cdots + \alpha_{k-1}t^k$ and hence $\mc{U}_k$ is unipotent. The coordinate choice $t$ gives sections of $\Diff_{k+1} \rightarrow \Diff_{k}$ for $k > 0$ sending $f(t) = \alpha_0t^1+ \cdots + \alpha_{k-1}t^k$ on the corresponding element in $\Diff_{k+1}$ with $\alpha_k=0$. We get an isomorphism of $\mc{U}_{k+1}/\mc{U}_k = \Diff_{k+1}/\Diff_{k}$ on $\bG_a$ showing $\Diff_{k}$ is a solvable group of unipotent radical $\mc{U}_{k}$, and a splitting of \ref{suite:unipotent} sending $\lambda \in \bG_m$ on $\theta_t(\lambda) = \lambda t$. The image of $\theta_t$ in $\Diff_{k}$ is a maximal torus and each maximal torus in $\Diff_k$ appears this way. 

For $\Diff_k = \bG_m\ltimes \mc{U}_k$ the action of $\bG_m$ on $f(t) = t + \alpha_1t^2 + \cdots + \alpha_{k-1}t^k \in \mc{U}_k$ is written
\begin{displaymath}
  \theta_t(\lambda)^{-1}f(t) \theta_t(\lambda) = t + \chi^{-1}(\lambda) \alpha_1t^2 + \chi^{-2}(\lambda) \alpha_2t^3 + \cdots + \chi^{-(k-1)}(\lambda)t^k.
\end{displaymath}
The weights of the $\bG_m$-action on $\mc{U}_k$ are $\chi^{-1}$, $\chi^{-2}$, $\ldots$, $\chi^{-(k-1)}$. For $\Diff_k = \mc{U}_k\rtimes \bG_m$ one gets the weights $\chi^{1}$, $\chi^{2}$, $\ldots$, $\chi^{(k-1)}$.  More generally if $u_1$, $u_2$, $\ldots$, $u_{k}$ are coordinates of $\mc{U}_k$ such that $u_\ell$ is non-zero in $\bG_a = \mc{U}_{\ell}/\mc{U}_{\ell-1}$ the group law on $\bG_m\ltimes \mc{U}_k$ is given by the co-multiplication $\mu^*$ on $\sO_s[M]\otimes \sO_s[u_1, \ldots, u_k]$ such that 
\begin{equation}
  \label{mucolaw}
  \mu^*\mid \sO_s[M]   =  1\otimes 1\quad \textrm{and} \quad \mu^*(u_\ell) = u_\ell\otimes 1 + \chi^{-\ell} \otimes u_\ell + (\star\star)_\ell
\end{equation}
where $(\star\star)_\ell$ is a polynomial in $u_m\otimes 1$, $1\otimes u_m$ for $1\leq m <\ell$ and $\chi^{-m}$ for $1 \leq m \leq \ell$. This is a special case of \cite[4.2 (3)]{KKMS} for a solvable group. If we write $\Diff_k = \mc{U}_k \rtimes \bG_m$ we have to replace $\chi^{-\ell}$ by $\chi^{\ell}$ in the group law expressions of \ref{mucolaw}.

\subsection{Toroidal embeddings of $\Diff_{k}$}
\label{subsec:Psivartheta}

There are two toric embeddings of $\bG_m$ corresponding to the subcones of $N$ given by $\sigma_+ = [0, +\infty[$ and $\sigma_-=]-\infty, 0]$. If $S_{\sigma} = \{m \in M ; m\mid \sigma \geq 0\}$ these embeddings are respectively given by the spectra of
\begin{displaymath}
  \bs{A}[S_{\sigma_+}] = \bigoplus_{\gamma \in \sigma_+} \bs{A}\chi^{\gamma} \quad \text{and} \quad \bs{A}[S_{\sigma_-}] = \bigoplus_{\gamma \in \sigma_-} \bs{A}\chi^{\gamma}. 
\end{displaymath} 
Following techniques of \cite[4]{KKMS} one can give equivariant embeddings of $\Diff_k$ in affine schemes. Let $\sigma$ be either $\sigma_+$ or $\sigma_-$ and write $\bG_m \hookrightarrow \bA_{\sigma} =\bA^1$ for the attached toric embedding. One can get an affine equivariant embedding of $\Diff_k$ by looking at $C_{\sigma} = \bA_{\sigma} \times^{\bG_m} \Diff_k$. By definition this is the quotient of $\bA_{\sigma} \times (\bG_m\ltimes \mc{U}_k)$ by the relation $[\lambda\cdot x, f(t)] = [x, \theta_t(\lambda)f(t)]$. As an algebraic variety $C_{\sigma} = \bA_{\sigma}\times \mc{U}_k$. It is clear that the right action of $\Diff_k$ on itself extends to $C_{\sigma}$. The left action extends to $C_{\sigma}$ if and only if $\sigma = \sigma_-$ : the affine embedding $\Diff_k \subset C_{\sigma}$ is given by 
\begin{displaymath}
  \bs{A}[S_{\sigma}]\otimes \sO_s[u_1, \ldots, u_k] \subset \bs{A}[M]\otimes \sO_s[u_1, \ldots, u_k],
\end{displaymath} 
looking at the group law \ref{mucolaw} one sees the left action of $\Diff_k$ on itself extends if and only if $S_{\sigma} \ni \chi^{-\ell}$ and thus $\sigma = \sigma_-$. Proceeding the same way (exchanging rights and lefts) with $E_{\sigma} = \Diff_k\times^{\bG_m} \bA_{\sigma}$ we get two affine \emph{bi-equivariant} embeddings $C_{\sigma_-}$ and $E_{\sigma_+}$ (for short $C$ and $E$) of $\Diff_k$. 

These constructions depend on the coordinate $t$. Let's focus on the case of the embedding $C$, the $E$ case is identical. Any other coordinate at $p$ is written $\vartheta t$ for $\vartheta \in \Diff_k$. If $\mbb{T}$ is the maximal torus in $\Diff_k$ defined by $t$ then $\vartheta\mbb{T}\vartheta^{-1}$ is the one defined by $\vartheta t$. Let $C_{t}$ and $C_{\vartheta t}$ be the affine bi-equivariant embeddings of $\Diff_k$ defined by $\sigma_-$. In $C_{\vartheta t}$ we have the relation $[\lambda^{-1}x, f(t)] = [x, \vartheta\lambda\vartheta^{-1}f(t)]$. By \cite[4 (iv)]{KKMS} page $183$, there is a unique bi-equivariant isomorphism $\Psi_{\vartheta} : C_{t} \simeq C_{\vartheta t}$ commuting to the embeddings of $\Diff_k$ in $C_{t}$ and $C_{\vartheta t}$. The isomorphism $\Psi_{\vartheta}$ is the composition of the left action $L_{\vartheta^{-1}} : C_{t} \rightarrow C_{t}$ with the isomorphism $\delta_{\vartheta} : C_{t} \rightarrow C_{\vartheta t}$ given by $[x, f(t)] \mapsto [x, \vartheta f(t)]$. Notice $\Psi_{\vartheta}$ is the identity on $\Diff_k$.  

\subsection{Back to twisted polar parts} Let $(\pi :\mc{X} \rightarrow S, D)$ be a stable genus $g$ marked curve over a scheme $S$. Following the previous procedure one gets two affine cones locally on the base. To glue these into a stack over $\fMm_{g, n}$ one needs to specify transition functions on the intersection of two open affine subsets of the base. Let $U_{12}$ be the intersection of two such affine subsets $U_1$ and $U_2$ and let $t_1$ and $t_2$ be coordinates at $D$ respectively over $U_1$ and $U_2$. Looking at $t_1$ and $t_2$ as elements of $\Diff_k$ there exists $\vartheta \in \Diff_k$ such that $t_2 = \vartheta t_1$. The transition function from $C_1$ on $C_2$ should correspond to the polar part coordinate change  
\begin{displaymath}
  \sum_{0<\ell \leq k} \frac{a_\ell}{t_1^{\ell}}  = \sum_{0 < \ell \leq k} \frac{a_\ell}{(\vartheta^{-1}t_2)^{\ell}}.
\end{displaymath}
This is exactly the right multiplication $R_{\vartheta^{-1}}$ on $\Diff_k$. This isomorphism extends on $C_1$ and $C_2$ by $G_{\vartheta} = \Psi_{\vartheta}R_{\vartheta^{-1}}$. To keep track of the polar parts we need to mod out by the left action of $\bs{\mu}_k \subset \bG_m$ on both $C_1$ and $C_2$. Fortunately, $G_{\vartheta}$ is $\bs{\mu}_{k}$-equivariant and this local construction together with the transition functions $G_{-}$ give a stack over $S$. This procedure still makes sense in the case of $n$-marked points and gives a locally affine stack $\mc{C}_{\vk}$ over $\fMm_{g, n}$. 
\begin{definition}
  The stack of twisted polar parts over $\fMm_{g, n}$ is the locally affine cone $\mc{C}_{\vk}$ defined previously. 
\end{definition}
Using the same transition functions we can glue the local cone $E$ into a stack $\mc{E}$ over $\fMm_{g, n}$. However, $\mc{E}$ doesn't correspond to the stack $\bs{P}_{g, \vk}$ of twisted polar parts introduced in \cite{ELSV}. The reason is simple, locally at the $i$-th marked point the right action of $\bG_m$ on $C = \bA_{\sigma_-} \times \mc{U}_k$ has weights $\chi^{-1}, \chi^{-1}, \chi^{-2}, \ldots, \chi^{-(k_i-1)}$ which is (up to isomorphism) the expected behavior for the right action of $\bG_m$ on twisted polar parts at the same marked point. This is not the case for the right action of $\bG_m$ on $\mc{E}$ which has weights equal to $\chi$.

\section{Polar curves}
\label{subsec:polarcurves}

\subsection{Stable polar curves}

Let $(\pi:\mc{X} \rightarrow S, \underline{D})$ be a prestable $n$-marked curve of genus $g$. The sheaf 
\begin{displaymath}
\sO_{\mc{X}}(\sum_{i=1}^n k_iD_i)/\sO_{\mc{X}} = \bigoplus_{i=1}^n \sO_{\mc{X}}(k_iD_i)/\sO_{\mc{X}}
\end{displaymath}
has support on the scheme defined by $\underline{D}$. Base change in cohomology for finite morphisms ensures that the pushforward  
\begin{displaymath}
\sP_{\mc{X}, \vk} = \bigoplus_{i=1}^n\pi_*(\sO_{\mc{X}}(k_iD_i)/\sO_{\mc{X}}) = \bigoplus_{i=1}^n \sP_{\mc{X}, k_i} 
\end{displaymath}
 is locally free on $S$ of rank $d=\sum_{i=1}^n k_i$. Let $p$ be a point in the support of $D_i$ and $t$ a local coordinate at $p$. For $s = \pi(p)$ any section $\varrho$ of $\sP_{\mc{X}, k_i}$ can be locally written in the form
\begin{equation}
  \label{eq:locpolarpart}
  \varrho_i = \frac{a_{k_i}}{t^{k_i}} + \cdots \frac{a_1}{t} \quad \text{with $a_{\ell} \in \sO_s$}. 
\end{equation}
With the difference with other coefficients $a_{k_i}$ is of global nature. The quotient morphism 
\begin{displaymath}
\begin{tikzpicture}[baseline=(current  bounding  box.center),>=latex]
    \matrix (m) 
    [matrix of math nodes, row sep=1em, column sep=2.5em, text
    height=1ex, text depth=0.25ex]  
    { \sO_{\mc{X}}(k_iD_i)/\sO_{\mc{X}} &
      \sO_{\mc{X}}(k_iD_i)/\sO_{\mc{X}}((k_i-1)D_i) \\ };  
    \path[->,font=\scriptsize]  
    (m-1-1) edge (m-1-2);
  \end{tikzpicture}
\end{displaymath}
has kernel supported on $D_i$. Its higher pushforwards are $0$ and one gets an onto map 
\begin{displaymath}
\begin{tikzpicture}[baseline=(current  bounding  box.center),>=latex]
    \matrix (m) 
    [matrix of math nodes, row sep=1em, column sep=2.5em, text
    height=1ex, text depth=0.25ex]  
    { \sP_{\mc{X}, k_i} & \sL_i^{-k_i}
      \\ };   
    \path[->>,font=\scriptsize]  
    (m-1-1) edge node[auto] {$\vartheta_{\mc{X}, k_i}$} (m-1-2);
  \end{tikzpicture}
\end{displaymath}
on the $(-k_i)$-th tensor power of the restriction of the cotangent bundle at the $i$-th marked point. The morphisms $\vartheta_{-,k_i}$ commute to base change for each domain and target does and the quotient morphisms do. The global picture we get is that of a locally free sheaf $\sP_{g, \vk}$ given by the collection $\sP_{-, \vk}$ and morphisms $\vartheta_{k_i} : \sP_{g, k_i} \rightarrow \sL_i^{-k_i}$ over the stack of prestable marked curves $\fM_{g, n}^{p\! s}$. 
\begin{definition}
  A polar $S$-curve $((\pi : \mc{X} \rightarrow S,\underline{D}), \underline{\varrho})$ is composed of a marked prestable curve and a global section $\underline{\varrho}=(\varrho_1, \ldots, \varrho_n)$ of $\sP_{\mc{X}, \vk}$ such that $\vartheta_{\mc{X}, k_i}(\varrho_i)$ trivializes $\sL_i^{-k_i}$.  
\end{definition}
Most of the time we will abuse the notation by writing $(\pi : \mc{X} \rightarrow S, \underline{\varrho})$ the previous polar curve. The condition ``$\vartheta_{\mc{X}, k_i}(\varrho_i)$ trivializes $\sL_i^{-k_i}$'' means $a_{k_i}$ (in expression \ref{eq:locpolarpart}) is invertible in any local representation of $\varrho_i$ at $p_i$. We say $\varrho_i$ is of order $k_i$ at $p_i$.
\begin{definition}
  A morphism from a polar curve $(\pi : \mc{X} \rightarrow S, \underline{\varrho})$ on $(\psi : \mc{Y} \rightarrow T, \underline{\varepsilon})$ over $f : S \rightarrow T$ is an isomorphism of the underlying marked curves $\psi$ and $f^*\pi$ sending $\underline{\varrho}$ on $f^*\underline{\varepsilon}$.  
\end{definition}
It is straightforward and formal to show that the functor in groupoids defined by the collection of polar curves over $\Sch$ is an algebraic stack. It is not separated, we do the obvious thing : restrict ourselves to objects having finite automorphism groups.  
\begin{lemma}
  \label{lem:stability}
  Let $(X, \underline{\varrho})$ be a polar curve over $\bbk$. The following are equivalent
  \begin{enumerate}
  \item the automorphism group of $(X, \underline{\varrho})$ is finite,
  \item the sheaf $\omega_X(\sum_{i=1}^nm_ip_i)$ is ample for $m_i \geq 2$,
  \item the sheaf $\omega_X(\sum_{i=1}^n2p_i)$ is ample,
  \item each irreducible component of $X$ has at least $3$ nodal or marked points except for \emph{possible} smooth rational components having only $1$ marked point and $1$ nodal point. 
  \end{enumerate}
\end{lemma}
\begin{proof}
  $(3) \Rightarrow (2)$ is clear. Proving $(2)\Rightarrow (4)$ and $(4)\Rightarrow (3)$ boils down to a simple degree count of the restriction $\omega_\ell(\sum_{i=1}^n m_ip_i)$ of $\omega_X(\sum_{i=1}^n m_ip_i)$ to each irreducible component $X_\ell$. Both result from the relation   
  \begin{equation}
    \label{eq:ample}
    \deg\big(\omega_\ell(\sum_{i=1}^nm_ip_i)\big)  = 2g_{X_\ell} -2 +
    \big(\sum_{ \{i \mid p_i \in X_\ell\} }m_i \big)+ \sharp{\{e \in
      X_\ell\mid e \textrm{ nodal point of }\, X\}}.
  \end{equation}
Let us focus on $(1) \Leftrightarrow (4)$. Assume $(X, \underline{p})$ satisfies $(4)$. The only components of $(X, \underline{p})$ that have infinite automorphism groups are genus $0$ components $(R, p, e)$ having marked point $p$ and nodal point $e$. It is enough to show the automorphism group of $(R,p, e)$ endowed with a polar part $\varrho$ at $p$ are finite. Identify $(R, p, e)$ with $(\bP^1, \infty, 0)$ and let $t_\infty$, $t_0$ be the coordinates of $\bP^1$. An automorphism of $(\bP^1, \infty, 0)$ is an element in $\bG_m$ sending $t_0$ on $\lambda t_0$. The polar part $\varrho$ at $p$ is written in the chart $\bP^1\setminus\{0\}$ as
\begin{displaymath}
  \varrho = \sum_{\ell=1}^k \frac{a_\ell}{t^{\ell}}, \quad \textrm{with $a_k \neq 0$ and $t = \frac{t_\infty}{t_0}$}. 
\end{displaymath}
If $\lambda \in \bG_m$ is an automorphism of $(\bP^1, \infty, 0, \varrho)$ we have that for all $\ell \in \{1, \dots, k\}$, $a_\ell = a_\ell\lambda^{\ell}$. For $\ell = k$ we get $a_k = \lambda^k a_k$ and $\lambda$ is therefore a root of unity. Conversely, assume the automorphism group of $(X, \underline{\varrho})$ is finite. There is nothing to prove if $(X, \underline{p})$ is stable. Let $(R, q_1, q_2)$ be an unstable component of $X$ with $q_1$ and $q_2$ either marked or nodal points. Then $(R, q_1, q_2)$ has exactly $1$ marked point and $1$ nodal one. If this wasn't the case then either $q_1$ and $q_2$ are both nodal points and the infinite automorphism group of $(R, q_1, q_2)$ is a subgroup of $\Aut(X, \underline{\varrho})$ or $q_1$ and $q_2$ are both marked points and for connectivity reasons $(X, \underline{p}) = (R, q_1, q_2)$. This case is excluded for $(g, n) \neq (0,2)$. Other cases with less marked or nodal points are treated similarly.
\end{proof}
\begin{definition}
  A \emph{stable} polar curve $(X, \underline{\varrho})$ over $\bbk$ is a polar curve satisfying one of the conditions of lemma \ref{lem:stability}.
\end{definition}
We show that stable polar curves form an open substack of the one of polar curves. We first make a statement on \emph{prestable curves} underlying a stable polar curve.
\begin{definition}
  A bubbly curve is a prestable curve $(X, \underline{p})$ satisfying the conditions (2) to (4) of lemma \ref{lem:stability}. An unstable component of $(X, \underline{p})$ is called a bubble.
\end{definition}
Let $(\pi : \mc{X} \rightarrow S, \underline{D})$ be a prestable curve whose fibers are bubbly curves. Having fibers that are bubbly curves is stable under base change. Indeed, this means the invertible sheaf $\omega_{\mc{X}_s}(\sum_{i=1}^n 2D_{i,s})$ is ample on each fiber $X_s$ for $s \in S$. Now ampleness and $\omega_{\mc{X}}(\sum_{i=1}^n2D_{i})$ commute to base change. Furthermore, relative ampleness is an \'etale-local property on the base (\cite[III.3.7]{MilneEtCoh}). Thus having bubbly fibers is \'etale-local on the base. Both remarks ensure that the functor in groupoids $\fB_{g, n}$ sending a scheme $S$ on the groupoid of prestable curves having bubbly fibers is a substack of $\fM_{g, n}^{p\! s}$.
\begin{proposition}
  \label{prop:Bgn}
  $\fB_{g, n}$ is an open (and thus algebraic) substack of $\fM_{g, n}^{p\! s}$.
\end{proposition}
\begin{proof}
  Let $(X, \underline{p})$ be a $\bbk$-point of $\fB_{g, n}$ and $(\pi: \mc{X} \rightarrow S, \underline{D})$ be a versal deformation of $(X, \underline{p})$ as a prestable curve. Since $(X, \underline{p})$ is a bubbly curve $\omega_{\mc{X}}(\sum_{i=1}^n 2D_i)_{\mid X} = \omega_X(\sum_{i=1}^n 2p_i)$ is ample. By Grothendieck existence theorem \cite[5.4.5]{EGA31} a formal versal deformation of $(X, \underline{p})$ is algebraizable and we get a family $(\pi^a : \mc{X}^a \rightarrow S^a, \underline{D}^a)$ over a complete henselian local ring extending $\pi$. Furthermore, the invertible sheaf $\omega_{\mc{X}^a}(\sum_{i=1}^n 2D_i^a)$ corresponding to this algebraization is relatively ample. We thus get a family of bubbly curves over an \'etale-local neighborhood of $(X, \underline{p})$ in the stack of prestable curves. This means $\fB_{g, n}$ is open in $\fM_{g, n}^{p\! s}$.
\end{proof} 
The previous definition implies families of polar curves having fibers that are stable polar curves define an open (algebraic) substack $\fP_{g, \vk}$ of the stack of polar curves.   
\begin{proposition}
  \label{prop:tame}
  $\fP_{g, \vk}$ is a tame stack in the sense of \cite{AOVtame}. It is a Deligne--Mumford stack if the characteristic of the base field $\bs{p}$ is either $0$ or does not divide any of the $k_i$s.
\end{proposition}
\begin{proof}
  Following \cite[8.1]{Laumon} to show $\fP_{g, \vk}$ is a Deligne--Mumford stack it is enough to show its diagonal is formally unramified (\cite[4]{DMMg}). Since $\fB_{g, n}$ and $\fP_{g, \vk} \rightarrow \fB_{g,n}$ are locally of finite presentation it is the case for $\fP_{g, \vk}$ as well. It is therefore enough to show fibers of closed points on the diagonal are discrete and reduced. These fibers are automorphism groups of points of $\fP_{g, \vk}$, they are quasi-finite. To show they are reduced it is enough to check automorphism groups over $S_\epsilon = \Spec(\bbk[\epsilon]/(\epsilon^2))$ are trivial. Let $(\mc{X} \rightarrow S_\epsilon, \underline{\varrho})$ be such an object. Automorphisms of $(\mc{X}, \underline{D})$ are given by infinitesimal vector fields on the central fiber $(X, \underline{p})$ fixing the marking. There are no such vector fields on stable marked curves (\cite[1.4]{DMMg}), we only have to check the case of a bubble $(R, p, e)$. Vector fields on $R$ fixing $p$ and $e$ are classified by $\Hi^0(T_R(-p-e))$ which is of dimension $1$. Identify $(R, p, e)$ with $(\bP^1, \infty, 0)$. A vector field on $(\bP^1, \infty, 0)$ is of the form $\lambda t\frac{\partial}{\partial t}$ where $t$ is the local parameter at $\infty$ and $\lambda \in \bbk$. A polar part of order $k$ at $p$ is locally of the form $\varrho = \sum_{\ell=1}^{k}a_\ell t^{-\ell}$. The vector field acts on $\varrho$ through the automorphism $\alpha = \id + \epsilon\partial$. Asking that $\alpha$ fixes $\varrho$ means 
\begin{displaymath}
  \sum_{\ell=1}^{k_p} \frac{a_\ell}{t^\ell}  = \sum_{\ell=1}^{k_p} \frac{a_\ell}{t^\ell+\epsilon\partial(t^\ell)} = \sum_{\ell=1}^{k_p} \frac{a_\ell}{t^\ell( 1 +\lambda\ell\epsilon)}  = \sum_{\ell=1}^{k_p} \frac{a_\ell}{t^\ell} - \lambda\epsilon \sum_{\ell=1}^{k_p} \frac{\ell a_\ell}{t^\ell}.
\end{displaymath}
Hence $\partial \neq 0$ if and only if $\bs{p} > 0$ and $\ell$ is a multiple $\bs{p}$ for every $a_\ell \neq 0$. In other words $\varrho$ is a $\bs{p}$ power. In this case the automorphism group scheme of $(\bP^1, \varrho)$ is a $\bs{\mu}_{\bs{p}^m}$. The automorphism group of a polar curve having vector fields is thus a product of finite tame groups by a number of $\bs{\mu}_{p^\ell}$. This is a linear reductive group and tameness follows from \cite[3.2]{AOVtame}. If either $\bs{p} \nmid k$ or $\bs{p} = 0$ then $\partial = 0$ because $a_k \neq 0$. If this is the case for every $k_i$ then $\fP_{g, \vk}$ is Deligne--Mumford. 
\end{proof}

\subsection{Structure of $\fP_{g, \vk}$ over $\fMm_{g, n}$}

There is a natural forgetful onto morphism $\fP_{g, \vk} \rightarrow \fB_{g, n}$ sending a stable polar curve on the underlying bubbly curve. Composing with the stabilization morphism $\rho : \fB_{g,n} \rightarrow \fMm_{g,n}$ (see \cite{BehManin}) contracting possible bubbles one gets an onto morphism $\varkappa : \fP_{g, \vk} \rightarrow \fMm_{g, n}$. Our goal is to study $\varkappa$, we start by studying $\rho$. 
\begin{remark}
We shall often restrict our study to the $1$-marked case. This is due to the following fact: let $\fP_{g, k_i, n-1}$ be the fiber product $\fP_{g, k_i} \times_{\fMm_{g, 1}} \fMm_{g, n}$ where $\fP_{g, k_i} \rightarrow \fMm_{g, 1}$ is the forgetful map $\varkappa$ in the $1$-marked case and $\fMm_{g, n} \rightarrow \fMm_{g, 1}$ keeps only the $i$-th marked point and stabilizes the resulting marked curve. Points of $\fP_{g, k_i, n-1}$ are marked bubbly curves having one polar part along $p_i$ and whose bubbles are marked by $p_i$. The fact is that
\begin{equation}
  \label{eq:isoprodfibre}
  \fP_{g, \vk} \simeq \fP_{g, k_1, n-1}\times_{\fMm_{g, n}}\fP_{g, k_2, n-1} \times_{\fMm_{g, n}}  \cdots \times_{\fMm_{g,n}} \fP_{g, k_n, n-1}. 
\end{equation}
The morphism from $\fP_{g, \vk}$ to the right-hand side comes from the forgetful morphisms keeping the $i$-th polar part and contracting bubbles not containing $p_i$.
\end{remark}
\begin{lemma}
  \label{lem:fonctcontraction}
  Let $(\pi : \mc{X} \rightarrow S, \underline{D})$ be a prestable curve and $(\pi^\star:
  \mc{X}^\star \rightarrow S, \underline{D}^\star)$ its stable model. For $\underline{m} = (m_1, \ldots, m_n) \in \left(\N^*\right)^n$ the stabilization morphism gives natural maps $\alpha_i : \mc{N}_{m_i D_i} \rightarrow \mc{N}_{m_i D_i^\star}$. 
\end{lemma}
\begin{proof}
  Since $\rho$ commutes to the markings defined by $D_i$ and $D_i^\star$ we have $D_i \subset \rho^{-1}D_i^\star$. Thus $\rho$ defines an $\sO_{\mc{X}}$-morphism $\rho^*\sO_{\mc{X}^\star}(-m_iD_i^\star) \rightarrow \sO_{\mc{X}}(-m_iD_i)$. 
Dualizing we get the morphism
\begin{displaymath}
  \begin{tikzpicture}[baseline=(current  bounding  box.center),>=latex]
    \matrix (m) 
    [matrix of math nodes, row sep=2.5em, column sep=2em, text height=1ex, text depth=0.25ex] 
    { \sO_{\mc{X}}(m_iD_i) & \Hhom_{\sO_{\mc{X}}}(\rho^*\sO_{\mc{X}^\star}(-m_iD_i^\star), \sO_{\mc{X}})\\};  
    \path[->,font=\scriptsize]  
    (m-1-1) edge (m-1-2);
  \end{tikzpicture}
\end{displaymath} 
which gives by adjunction 
\begin{displaymath}
\begin{tikzpicture}[baseline=(current  bounding  box.center),>=latex]
    \matrix (m) 
    [matrix of math nodes, row sep=2.5em, column sep=2em, text height=1ex, text depth=0.25ex] 
    { \rho_*\sO(m_iD_i) & \Hhom_{\sO_{\mc{X}}^\star}(\sO_{\mc{X}^\star}(-m_iD_i^\star), \rho_*\sO_{\mc{X}}).\\};  
    \path[->,font=\scriptsize]  
    (m-1-1) edge (m-1-2);
  \end{tikzpicture}
\end{displaymath}
The term on the right is isomorphic to $\sO_{\mc{X}^\star}(m_iD_i^\star)$ for $\rho_*\sO_{\mc{X}}= \sO_{\mc{X}^\star}$. We get the desired map by moding out by $\sO_{\mc{X}} \subset \sO_{\mc{X}}(m_iD_i)$ and $\sO_{\mc{X}^\star} \subset \sO_{\mc{X}}^\star(m_iD_i^\star)$.
\end{proof} 
Following \cite[2, (II)]{Knudsen2} of proof 2.4, the bubbly curve $(\pi : \mc{X}\rightarrow S, \underline{D})$ can be recovered as the relative $\Proj$ over $\mc{X}^\star$ of $\rho_*\sO_{\mc{X}}(\sum_{i=1}^n D_i)$. The $n$ marked $S$-points are given by 
\begin{displaymath}
(\rho_*\sO_{\mc{X}}(D_i))\mid_{D_i} \twoheadrightarrow \rho_*\mc{N}_{D_i}
\end{displaymath}
where $\rho_*\mc{N}_{D_i}$ is invertible because it is supported on $D_i$ and thus is equal to $(\rho\mid_{{\cD_i}})_*\mc{N}_{D_i}$. The starting point to the study of $\rho$ is the following fact: we have a commutative diagram 
\begin{equation}
  \label{diagmu}
  \begin{tikzpicture}[baseline=(current  bounding  box.center),>=latex]
    \matrix (m) 
    [matrix of math nodes, row sep=2.5em, column sep=2em, text
    height=1ex, text depth=0.25ex]  
    {   0  & \sO_{\mc{X}^\star} & \rho_*\sO_{\mc{X}}(\sum_{i=1}^nD_i) & \bigoplus_{i=1}^n\rho_*\mc{N}_{D_i}  & 0\\
        0  & \sO_{\mc{X}^\star} & \sO_{\mc{X}^\star}(\sum_{i=1}^nD_i^\star) & \bigoplus_{i=1}^n\mc{N}_{D_i^\star} & 0\\};  
    \path[->,font=\scriptsize]  
    (m-1-1) edge (m-1-2)
    (m-1-2) edge (m-1-3)
    (m-1-3) edge (m-1-4)
    (m-1-4) edge (m-1-5)
    (m-2-1) edge (m-2-2)
    (m-2-2) edge (m-2-3)
    (m-2-3) edge (m-2-4)
    (m-2-4) edge (m-2-5)
    (m-1-3) edge node[auto] {$\hat{\alpha}$} (m-2-3)
    (m-1-4) edge node[auto] {$\alpha$} (m-2-4);
    \draw[double, double distance=1.5 pt, font=\scriptsize]
    (m-1-2) -- (m-2-2);
  \end{tikzpicture}
\end{equation}
where $\hat{\alpha}$ is induced by $\rho$ and $\alpha$ is the direct sum of the $\alpha_i$s of lemma \ref{lem:fonctcontraction} for $\underline{m} =(1, \ldots, 1)$. This diagram means $\rho_*\sO_{\mc{X}}(\sum_{i=1}^nD_i)$ is the Yoneda pullback of the low horizontal extension by $\alpha$, i.e. 
\begin{displaymath}
\rho_*\sO_{\mc{X}}\Big(\sum_{i=1}^nD_i\Big) \simeq \sO_{\mc{X}^\star}\Big(\sum_{i=1}^n D_i^\star\Big) \times_{\mc{N}_{\sum_{i=1}^n D_i^\star}} \rho_*\mc{N}_{\sum_{i=1}^n D_i}
\end{displaymath}
and $\hat{\alpha}$ is given by the projection on the first factor. The only point to clarify about diagram \ref{diagmu} is the right exactness of the first row. The morphism $\rho$ satisfies $\rho_*\sO_{\mc{X}} = \sO_{\mc{X}^\star}$ and its fibers are rational $1$-dimensional curves. In particular $\Hi^1(\rho^{-1}(x), \sO_{\rho^{-1}(x)})= 0$ for every point $x\in \mc{X}^\star$ and following \cite[1.5]{Knudsen2} we get that $\Ri^1\rho_*\sO_{\mc{X}} = 0$.
\begin{notation}
\label{nota:strctcone}
Given morphisms $\alpha_i: \ms{R}_i \rightarrow \bigoplus_{i=1}^n \mc{N}_{D_i^\star}$ of invertible $\sO_{\mc{X}^\star}$-modules write 
\begin{displaymath}
  \sE_\alpha= \sO_{\mc{X}^\star}\Big(\sum_{i=1}^nD_i^\star\Big) \times_{\mc{N}_{\sum_{i=1}^n D_i^\star}} \Big(\bigoplus_{i=1}^n\ms{R}_i\Big) 
\end{displaymath}
and $D_{\alpha, i}$ for the marked $S$-point defined by $\tau_{\alpha, i} : (\sE_\alpha)\mid_{D_i^\star} \twoheadrightarrow \mc{N}_{\cD_i^\star}$.
\end{notation} 
\begin{proposition}
  \label{prop:fiberBM}
  The fiber of $\rho$ over $(\pi^{\star}: \mc{X}^{\star} \rightarrow S, \underline{D}^\star)$ is isomorphic to the stack whose 
\begin{enumerate}
\item
objects over $T \in \SchS$ are morphisms $(\alpha_i : \ms{R}_i \rightarrow \mc{N}_{D_i^\star})_{i=1}^n$ of invertible sheaves,
\item
morphisms from $(\alpha_i :\ms{R}_i \rightarrow \mc{N}_{D_i^\star})_{i=1}^n$ et $(\beta_i : \ms{S}_i \rightarrow \mc{N}_{D_i^\star})_{i=1}^n$ over $f \in \Hom_S(T,T')$ are isomorphisms $\ms{R}_i \simeq f^*\ms{S}_i$ sending $\beta_i$ on $\alpha_i$.
\end{enumerate}
\end{proposition}
\begin{proof}
  We restrict the proof to the $1$-marked case. It is clear that points (1) and (2) define a category fibered in groupoids. For each $T$-object $(\pi : \mc{X} \rightarrow T, D)$ and $T'$-object $(\pi': \mc{X}' \rightarrow T', \sigma")$ in the fiber of $\rho^{-1}(\pi^\star : \mc{X}^\star \rightarrow S, D^\star)$ and each morphism $\psi$ from $(\pi', D')$ on $(\pi, D)$ over $f \in \Hom_S(T',T)$ (i.e. isomorphism $\psi : (\pi', D') \simeq f^*(\pi, D)$) define the functor (of categories fibered in groupoids) $\bs{F}$ as follows   
\begin{itemize}
\item $\bs{F}(\pi : \mc{X} \rightarrow T, D)$ is the morphism $\alpha_{\pi} : \rho_*\mc{N}_D \rightarrow \mc{N}_{D^\star}$ of lemma \ref{lem:fonctcontraction},
\item since $\psi$ fixes the marked point we get an isomorphism $\mu_{\psi} : \rho_*\mc{N}_{D'} \simeq \rho_*f^*\mc{N}_{D}$ sending $\alpha_{\pi'}$ on $\psi^*\alpha_{\pi}$. Since $\rho_*\mc{N}_{D}$ commutes to base change we get $\rho_*f^*\mc{N}_D \simeq f^*\rho_*\mc{N}_D$. We define $\bs{F}(\psi)$ as the morphism $(\mu_{\psi} : \rho_*\mc{N}_{D'} \simeq f^*\rho_*\mc{N}_{D})$.   
\end{itemize}
This functor is fully faithfull: an $\sO_{T'}$-isomorphism $\mu : \rho_*\mc{N}_{D'}\simeq f^*\rho_*\mc{N}_{D}$ for $T' \in \SchS$ defines by pulling-back Yoneda extensions an isomorphism 
\begin{displaymath}
  \alpha^{*_Y}\mu : \rho_*\sO_{\mc{X}'}(D') \simeq f^*\rho_*\sO_{\mc{X}}(D)
\end{displaymath}
such that $\mu^{-1} \tau_{\alpha'} = \tau_{\alpha} \alpha^{*_Y}\mu$. This is exactly saying the $(\mc{X}^\star, D^\star)$-isomorphism 
\begin{displaymath}
  \bP(\rho_*\sO_{\mc{X'}}(D'))\simeq \bP(f^*\rho_*\sO_{\mc{X}}(D)) = f^{-1}\bP(\rho_*\sO_{\mc{X}}(D))
\end{displaymath}
attached to $\alpha^{*_Y}\mu$ fixes the marked points. 

Let's show $\bs{F}$ is essentially surjective. Let $(\pi_\alpha : \mc{X}_\alpha\rightarrow T, D_\alpha)$ be the bubbly $T$-curve $(\pi_\alpha : \bP(\sE_\alpha)\rightarrow T, D_\alpha)$ attached to $\alpha : \ms{R} \rightarrow \mc{N}_{D^\star}$. We are going to show we have canonical isomorphisms $\sE_{\alpha}\simeq\rho_*\sO_{\mc{X}_\alpha}(D_\alpha)$ and $\ms{R} \simeq \rho_*\mc{N}_{D_\alpha}$ commuting to the quotient morphisms $\sE_\alpha \rightarrow \ms{R}$ and $\rho_*\sO_{\mc{X}_\alpha}(D_\alpha) \rightarrow \rho_*\mc{N}_{D_\alpha}$. The first isomorphism gives isomorphisms from $(\pi_\alpha:\mc{X}_\alpha \rightarrow T)$ on $(\bP(\rho_*\sO_{\mc{X}_\alpha}(D_\alpha))\rightarrow T)$ and the second says these fix $D_\alpha$. Let $\sigma_\alpha$ be the $S$-section defining $D_\alpha$ and given by $\tau_\alpha: \sE_\alpha \twoheadrightarrow (\sigma^\star)_*\ms{R}$. Consider the exact sequence 
\begin{equation}
  \label{suite:Ealpha}
  \begin{tikzpicture}[baseline=(current  bounding  box.center), >=latex]
    \matrix (m) 
    [matrix of math nodes, row sep=2.5em, column sep=2em, text height=1ex, text depth=0.25ex] 
    { 0 &  \sO_{\mc{X}_\alpha} & \rho^*\sE_\alpha & \rho^*\rho_*\sigma_{\alpha *}\ms{R}  & 0.\\};  
    \path[->,font=\scriptsize]  
    (m-1-1) edge (m-1-2)
    (m-1-2) edge (m-1-3)
    (m-1-3) edge (m-1-4)
    (m-1-4) edge (m-1-5);
  \end{tikzpicture}
\end{equation}
By definition $\rho^*\sE_\alpha$ is the invertible sheaf $\sO_{\bP(\sE_\alpha)}(1)$ on $\mc{X}_{\alpha}=\bP(\sE_{\alpha})$. Because $\sigma_{\alpha
  *}\ms{R}$ is supported on $D_\alpha$ we have that $\rho^*\rho_*\sigma_{\alpha *}\ms{R} = \sigma_{\alpha *}\ms{R}$. Now locally $\sigma_{\alpha *} \ms{R}
= \sO_{D_\alpha}$ and the $0$-Fitting ideal of the left morphism of the exact sequence \ref{suite:Ealpha} describes the equations of $D_\alpha$ in $\mc{X}_\alpha$. By definition of Fitting ideals these equations are given by the image of the morphism  
\begin{displaymath}
\begin{tikzpicture}[baseline=(current  bounding  box.center),>=latex]
    \matrix (m) 
    [matrix of math nodes, row sep=2.5em, column sep=2em, text height=1ex, text depth=0.25ex] 
    { \sO_{\mc{X}_\alpha}\otimes\big(\rho^*\sE_\alpha\big)^\vee
      = \big(\rho^*\sE_\alpha\big)^\vee & \sO_{\mc{X}_\alpha}.\\};  
    \path[->,font=\scriptsize]  
    (m-1-1) edge (m-1-2);
  \end{tikzpicture}
\end{displaymath}
attached to $\sO_{\mc{X}_\alpha} \rightarrow \rho^*\sE_{\alpha}$. This morphism is locally given by a generator of $\sO_{\bP(\sE_\alpha)}(-1) =
\rho^*\sE_\alpha^\vee$ and has image $\sO_{\mc{X}_\alpha}(-D_\alpha) \subset \sO_{\mc{X}_{\alpha}}$. We thus get an isomorphism $\rho^*\sE_\alpha = \sO_{\mc{X}_\alpha}(D_\alpha)$ and the injection $\sO_{\mc{X}_\alpha} \rightarrow \rho^*\sE_\alpha$ is the canonical map $\sO_{\mc{X}_\alpha} \rightarrow \sO_{\mc{X}_\alpha}(D_\alpha)$. Following the projection formula and the fact that $\rho_*\sO_{\mc{X}_\alpha} = \sO_{\mc{X}^\star}$ we get that $\sE_\alpha \simeq \rho_*\sO_{\mc{X}_\alpha}(D_\alpha)$. To conclude, we have got a commutative diagram giving the needed material to end the proof
\begin{displaymath}
\begin{tikzpicture}[baseline=(current  bounding  box.center),>=latex]
    \matrix (m) 
    [matrix of math nodes, row sep=2.5em, column sep=2em, text
    height=1ex, text depth=0.25ex]  
    {   0  & \sO_{\mc{X}^\star} & \sE_\alpha & \ms{R}  & 0\\
        0  & \sO_{\mc{X}^\star} &
        \rho_*\sO_{\mc{X}_\alpha}(D_\alpha) &
        \rho_*\mc{N}_{D_\alpha} & 0.\\};  
    \path[->,font=\scriptsize]  
    (m-1-1) edge (m-1-2)
    (m-1-2) edge (m-1-3)
    (m-1-3) edge (m-1-4)
    (m-1-4) edge (m-1-5)
    (m-2-1) edge (m-2-2)
    (m-2-2) edge (m-2-3)
    (m-2-3) edge (m-2-4)
    (m-2-4) edge (m-2-5);
    \draw[double, double distance=1.5 pt, font=\scriptsize]
    (m-1-3) -- (m-2-3)
    (m-1-4) -- (m-2-4)
    (m-1-2) -- (m-2-2);
  \end{tikzpicture}
\end{displaymath}
\end{proof}
\begin{corollary}
\label{cor:434}
We have an isomorphism 
\begin{displaymath}
\rho^{-1}(\pi^\star : \mc{X}^\star \rightarrow S,
\underline{D}^\star) \simeq \prod_{i=1}^n [\mbb{V}_S(\mc{N}_{D^\star_i})/\bG_m] 
\end{displaymath}
\end{corollary}
\begin{proof}
  Proposition \ref{prop:fiberBM} is the definition of the right-hand side of this isomorphism.
\end{proof} 
\begin{proposition}
\label{lem:retraductionproduitfibre}
The fiber of $\varkappa$ over $(\pi^\star : \mc{X}^\star\rightarrow S, \underline{D}^\star)$ is isomorphic to the stack whose
\begin{enumerate}
\item
objects over $T\in \SchS$ are collections of couples 
\begin{displaymath}
\Big(\alpha_i : \ms{R}_i \rightarrow \mc{N}_{D_T^\star}, \varrho_{i} \in
  \Hi^0\Big(T, \frac{\Sym^{k_i}(\sE_{\alpha_i})}{\sO_{\mc{X}_T^\star}\Delta_i^{k_i}}
  \Big) \Big)_{i=1}^n     
\end{displaymath}
where $\alpha_i \in \Hom_{\sO_T}(\ms{R}_i,\mc{N}_{D_T^\star})$ and $\ms{R}_i$ are $\sO_T$-invertible sheaves, $\Delta_i$ is the global section defined by the injection $\sO_{\mc{X}_T^\star} \rightarrow \sE_{\alpha_i}$ and the image of $\varrho_{i}$ trivializes 
 \begin{displaymath}
\frac{\Sym^{k_i}(\sE_{\alpha_i})}{\sO\Delta_i^{k_i}}\left\slash
\Delta_i\frac{\Sym^{{k_i}-1}(\sE_{\alpha_i})}{\sO\Delta_i^{{k_i}-1}}\right.   
\simeq \left(\frac{\sE_{\alpha_i}}{\sO\Delta_i}\right)^{\otimes {k_i}} = \mc{N}_{D_{\alpha_i}}^{\otimes k_i}. 
\end{displaymath}   
\item
morphisms from $(\alpha_i, \varrho_i)$ on $(\beta_i, \varepsilon_i)$ over $f \in \Hom_S(T, T')$ are collections of morphisms $\mu = (\mu_i)_i$ from $(\alpha_i)_i$ on $(\beta_i)_i$ in the sense of proposition \ref{prop:fiberBM} such that the induced morphisms $\Sym^k(\hat{\mu}_i) : \Sym^k(\sE_{\alpha_i})
\rightarrow f^*\Sym^k(\sE_{\beta_i})$ send $\varrho_{i}$ on $f^*\varepsilon_{i}$.  
\end{enumerate}
\end{proposition}
\begin{proof}
Following isomorphism \ref{eq:isoprodfibre} it is enough to treat the $1$-marked case. Write $V = \mbb{V}_S(\mc{N}_{D^\star})$ and $v$ for the standard atlas $V \rightarrow [V/\bG_m]$. Following proposition \ref{prop:fiberBM} and corollary \ref{cor:434} an $S$-morphism $F: T \rightarrow [V/\bG_m]$ corresponds to an $\sO_S$-morphism 
\begin{displaymath}
\begin{tikzpicture}[baseline=(current  bounding  box.center),>=latex]
    \matrix (m) 
    [matrix of math nodes, row sep=2.5em, column sep=2em, text height=1ex, text depth=0.25ex] 
    { \alpha_F : \mc{N}_{D^\star}\otimes \sO_T & \mc{N}_{D^\star} \otimes \sO_T.\\};  
    \path[->,font=\scriptsize]  
    (m-1-1) edge (m-1-2);
  \end{tikzpicture}
\end{displaymath}
 It is straightforward to see that $\fP_{g,k}\times_{\fB_{g,n}}\left[V/\bG_m\right]$ is isomorphic to the stack having objects over $T \in \SchS$  
\begin{displaymath}
  \big( \alpha : \ms{R} \rightarrow \mc{N}_{D^\star}\otimes \sO_T,\,\, \varrho \text{ a polar part of order $k$ on $\bP(\sE_{v})$}\big) 
\end{displaymath}
and morphisms from $(\alpha: \ms{R} \rightarrow \mc{N}_{D^\star}\otimes \sO_T, \varrho)$ on $(\beta : \ms{S} \rightarrow \mc{N}_{D^\star}\otimes \sO_{T'},\,\, \varepsilon)$ over $f : T\rightarrow T'$ corresponding to the $(\mc{X}^\star, D^\star)$-morphisms 
\begin{displaymath}
\begin{tikzpicture}[baseline=(current  bounding  box.center),>=latex]
    \matrix (m) 
    [matrix of math nodes, row sep=2.5em, column sep=2em, text height=1ex, text depth=0.25ex] 
    { \phi : (\bP(\sE_{\alpha}), D_{\alpha}, (T\rightarrow S)^*\varrho) & (\bP(\sE_{\beta}), D_{\beta},(T'\rightarrow S)^*\varepsilon).\\};  
    \path[->,font=\scriptsize]  
    (m-1-1) edge (m-1-2);
  \end{tikzpicture}
\end{displaymath}
 Now any such isomorphism $\phi$ can be recovered from the isomorphism $\mu_{\phi} : \alpha \simeq f^*\beta$ attached to $\phi$ and described in the proof of proposition \ref{prop:fiberBM}. The morphism $\mu_{\phi}$ defines by functoriality of Yoneda extensions an isomorphism $\hat{\mu}_{\phi} : \sE_{\alpha} \simeq f^*\sE_{\beta}$ that is equal to the one induced by $\phi$. Since $\sO_{\bP(\sE_{\alpha})}(D_{\alpha}) = \sO_{\bP(\sE_{\alpha})}(1)$ we have that $\rho_*\sO_{\bP(\sE_{\alpha})}(kD_{\alpha}) = \Sym^k(\sE_{\alpha})$ and equally for $\beta$. This means that $\Sym^k(\hat{\mu}_\phi)$ is the isomorphism 
\begin{displaymath}
\rho_*\sO_{\bP(\sE_{\alpha})}(kD_{\alpha})\simeq \rho_*\sO_{\bP(f^*\sE_{\beta})}(kD_{\beta})
\end{displaymath}
 induced by $\phi$ and that the space of polar parts on $\bP(\sE_{\alpha})$ is equal to $\Sym^k(\sE_{\alpha})/\sO_{\mc{X}_T^\star}\Delta^k$. Finally $\phi$ acts on polar parts by the induced morphism on quotients   
\begin{displaymath}
  \underline{\Sym^k}(\hat{\mu}_\phi) : \frac{\Sym^k(\sE_{\alpha})}{\sO_{{\mc{X}_T^\star}}\Delta^k} \simeq \frac{\Sym^k(f^*\sE_{\beta})}{\sO_{{\mc{X}_T^\star}}\Delta^k}.
\end{displaymath}
The only thing left to check is that the ``of order $k$'' condition corresponds to (1), but this is clear since 
\begin{displaymath}
\left(\frac{\sE_\alpha}{\sO\Delta}\right)^{\otimes k} = \frac{\Sym^k(\sE_\alpha)}{\sO\Delta^k}\left\slash
  {\Delta\frac{\Sym^{k-1}(\sE_\alpha)}{\sO\Delta^{k-1}}} \right. \simeq \frac{\sO_{\bP(\sE_{\alpha})}(kD_\alpha)}{\sO_{\bP(\sE_{\alpha})}}\left\slash \frac{\sO_{\bP(\sE_{\alpha})}((k-1)D_\alpha)}{\sO_{\bP(\sE_{\alpha})}}\right.= \mc{N}_{D_\alpha}^{k}. 
\end{displaymath}   
\end{proof}
We are going to explicit the implicit local picture in proposition \ref{lem:retraductionproduitfibre}. We start by giving a local presentation of $\sE_{\alpha}$. Assume $S=\Spec(\bs{A})$ for a local $\bbk$-algebra $\bs{A}$ and let $t$ be a local equation of $D^\star$ in a neighborhood $U=\Spec(\bs{B})$ of a closed point in the support of $D^\star$. Since $D^\star$ is \'etale on $S$ we identify $\bs{B}/t\bs{B}$ with $\bs{A}$ and write $c_0(-)$ the quotient map $\bs{B} \twoheadrightarrow \bs{A}$.
\begin{lemma}
  \label{lem:localEalpha}
  The restriction $\bs{E}_\alpha$ of $\sE_\alpha$ to $V$ has a presentation of the form 
\begin{displaymath}
\begin{tikzpicture}[baseline=(current  bounding  box.center),>=latex]
    \matrix (m) 
    [matrix of math nodes, row sep=2.5em, column sep=2em, text height=1ex, text depth=0.25ex] 
    {   0  & \bs{B} & \bs{B} \oplus \bs{B} & \bs{E}_\alpha  & 0\\}; 
    \path[->,font=\scriptsize]  
    (m-1-1) edge (m-1-2)
    (m-1-2) edge node[above] {$(-\alpha, t)$} (m-1-3)
    (m-1-3) edge node[above] {$\varphi$} (m-1-4)
    (m-1-4) edge (m-1-5);
  \end{tikzpicture}
\end{displaymath}
 where $\varphi(b, b') = (tb+\alpha b', c_0(b'))$. In particular, $\Sym_{\bs{B}}\bs{E}_\alpha = \frac{\bs{B}[Z_1, Z_2]}{(-\alpha Z_1 + t Z_2)}$.
\end{lemma}
\begin{proof}
Locally $\alpha$ is just an element of $\bs{A}$ and $\bs{E}_\alpha =
\bs{B}t^{-1}\times_{\bs{B}}\bs{A}$ is the set  
\begin{displaymath}
  \bs{E}_\alpha = \{(bt^{-1}, a) \in \bs{B}t^{-1} \times \bs{A}\mid c_0(b) = \alpha a\} \simeq \{ (b, a) \in \bs{B} \times \bs{A}\mid c_0(b) = \alpha a\}.
\end{displaymath}
An element $(b,a) \in \bs{E}_{\alpha}$ is of the form $(tb'+\alpha a, a)$ for $b' \in B$. This shows $\varphi$ is onto and sends $Z_1=(1, 0)$ (resp. $Z_2=(0,1)$) on $(t, 0)$ (resp. $(\alpha, 1)$). The kernel of $\varphi$ is composed of couples $(b,b')$ such that $b'=tb''$ and $tb+\alpha tb'' =0$ thus $b = -\alpha b''$. Finally $\Ker(\varphi) = \bs{B}(-\alpha, t)$. 
\end{proof}
\begin{proposition}
  \label{prop:localfinal}
  In the case $S$ is the spectrum of a local $\bbk$-algebra $\bs{A}$ the fiber of $\varkappa$ over $(\pi^\star: \mc{X}^\star \rightarrow S, \underline{D})$ is isomorphic to the quotient stack 
  \begin{displaymath}
    \prod_{i=1}^n[\bA_{\bs{A}}^{k_i}/\bs{\mu}_{k_i}] \rightarrow \bs{A}
  \end{displaymath}
  where $\bs{\mu}_{k_i}$ acts on $\bA_{\bs{A}}^{k_i}$ with weights $(1, 1, 2, \ldots, k_i-1)$.
\end{proposition}
\begin{proof}
  Write $V=\mbb{V}_S(\mc{N}_{D^\star})$ and $v$ for the standard atlas $V \rightarrow \left[V/\bG_m\right]$. In this local case $V \simeq \bA_{\bs{A}}^1$ and $\alpha_v$ is an element in $\bs{A}$. The diagram \ref{diagmu} corresponds to 
\begin{displaymath}
\begin{tikzpicture}[baseline=(current  bounding  box.center),>=latex]
    \matrix (m) 
    [matrix of math nodes, row sep=2.5em, column sep=2em, text
    height=1ex, text depth=0.25ex]  
    {   0  & \bs{B} & \bs{E}_v = \bs{B}t^{-1}\times_v \bs{A}
      & \mc{N}_{D_v}=\bs{A}  & 0\\
        0  & \bs{B} &
        \bs{B}t^{-1} &
        \frac{\bs{B}t^{-1}}{\bs{B}} = \bs{A} & 0.\\};  
    \path[->,font=\scriptsize]  
    (m-1-1) edge (m-1-2)
    (m-1-2) edge node[above] {$(t, 0)$} (m-1-3)
    (m-1-3) edge (m-1-4)
    (m-1-4) edge (m-1-5)
    (m-2-1) edge (m-2-2)
    (m-2-2) edge (m-2-3)
    (m-2-3) edge (m-2-4)
    (m-2-4) edge (m-2-5)
    (m-1-3) edge (m-2-3)
    (m-1-4) edge node[auto] {$\alpha_v$} (m-2-4);
    \draw[double, double distance=1.5 pt, font=\scriptsize]
    (m-1-2) -- (m-2-2);
  \end{tikzpicture}
\end{displaymath}
Let $Z_1=(t,0)$ and $Z_2 =(v, 1)$ be the elements of $\bs{E}_v$ given in lemma \ref{lem:localEalpha}. The filtration by order on the space of polar parts of order at most $k$ along $D_v$ corresponds to
\begin{displaymath}
  Z_1^{k-1}\frac{\bs{E}_v}{\bs{B}\cdot Z_1} \subset \ldots \subset
  Z_1^{\ell}\frac{\Sym^{k-\ell}\bs{E}_v}{\bs{B}\cdot Z_1^{k-\ell}} \subset
  \ldots \subset \frac{\Sym^k(\bs{E}_v)}{\bs{B}\cdot Z_1^k}.
\end{displaymath}
In particular
\begin{displaymath}
\forall \ell \in \{1, \ldots, k\}, \quad \left(\mc{N}_{\cD_v}\right)^{\otimes k-\ell} = Z_1^{\ell}\frac{\Sym^{k-\ell}\bs{E}_v}{\bs{B}\cdot
    Z_1^{k-\ell}}\left\slash Z_1^{\ell+1}\frac{\Sym^{k-\ell-1}\bs{E}_v}{\bs{B}\cdot Z_1^{k-\ell-1}}\right. = A\cdot Z_1^{\ell}Z_2^{k-\ell}.  
\end{displaymath}
Let's show that we have an $\bs{A}$-isomorphism 
\begin{equation}
\label{eq:isolocalEv}
\frac{\Sym^k(\bs{E}_v)}{\bs{B} \Delta^k} \simeq \bigoplus_{\ell =0}^{k-1} \mc{N}_{D_v}^{\otimes k-\ell}.
\end{equation}
We have to show that $Z_1^\ell Z_2^{k-\ell}$ for $\ell \neq k$ gives an $\bs{A}$-basis of the left-hand side. Following lemma \ref{lem:localEalpha} and since $\bs{B}\Delta$ is locally generated by $Z_1$ we get that 
\begin{displaymath}
\frac{\Sym^k(\bs{E}_v)}{\sO\Delta^k} = \frac{\bs{B}[Z_1, Z_2]_k}{\bs{B}Z_1^k + \bs{B}[Z_1, Z_2]_{k-1}(-vZ_1 + tZ_2)}.
\end{displaymath}
Now the right-hand side is generated by $Z_1^\ell Z_2^{k-\ell}$ for $\ell \neq k$ over $\bs{B}$. Since $vZ_1 = tZ_2$ the monomials $Z_1^\ell Z_2^{k-\ell}$ generate $\Sym^k(\bs{E}_v)/\bs{B}\Delta^k$ over $\bs{A}$. The latter is a flat $\bs{A}$-module of rank $k$ and $\{Z_1^\ell Z_2^{k-\ell}\}$ is therefore an $\bs{A}$-basis of $\Sym^k(\bs{E}_v)/\bs{B}\Delta^k$. 

Following proposition \ref{lem:retraductionproduitfibre} the fiber of $\rho$ over $(\pi^\star: \mc{X}^\star \rightarrow S, D^\star)$ is the stack whose 
\begin{enumerate}
\item objects are given by sections $(\alpha, a_k, \ldots, a_1) \in \mc{N}_{D_v}^{-1} \oplus \bigoplus_{\ell = 0}^{k-1}\mc{N}_{D_v}^{k-\ell}$ with $a_k \in \bs{A}^*$.
\item morphisms from $(\alpha, a_k, \ldots, a_1)$ on $(\alpha', a_k', \ldots,a_1')$ are each given by an element in $\bG_m$ acting diagonally with weight $j$ on an element in $\mc{N}_{D_v}^j$ and sending $(v,a_i)$ on $(v', a_i')$.
\end{enumerate}
Since $\mc{N}_{D_v}$ is isomorphic to $\bs{A}$ this means the previous fiber is the open substack of the weighted projective stack 
\begin{displaymath}
  \bP_{\bs{A}}(-1, k, k-1, \ldots, 1) = [\bA_{\bs{A}}^{k+1}/\bG_m]
\end{displaymath}
given by $a_k \neq 0$. This is known to be isomorphic to the quotient stack $[\bA_{\bs{A}}^k/\bs{\mu}_k]$ where $\bs{\mu}_k$ acts with weights $(-1, k-1, \ldots, 1)$. But this substack is isomorphic to the same quotient stack but where $\bs{\mu}_k$ acts with weights $(1,1, 2, \ldots, k-1)$.   
\end{proof}

\subsection{Relation to twisted polar parts}

We show that the stack of twisted polar parts $\mc{C}_{\vk}$ constructed in section \ref{subsec:Toroidal} is isomorphic to $\fP_{g, \vk}$ in case $\bs{p}$ is $0$ or greater than $d$. Since the former's construction is of local nature the expected isomorphism is built locally. We shall keep the notation introduced at the end of the previous section, we assume for instance that $S$ is the spectrum of a local $\bbk$-algebra $\bs{A}$. For notation and details regarding the stack of twisted polar parts refer to section \ref{subsec:Toroidal}. It is enough to study the $1$-marked case.
\begin{lemma}
  The action of $\Diff_k$ on the set of polar parts at $D$ lifts to an action on the fiber of $\varkappa$ over $(\pi^\star : \mc{X}^\star \rightarrow S, D)$ acting trivially on the zero section and having isoptropy groups $\bs{\mu}_k$ away from it.
\end{lemma}
\begin{proof}
  Let $\bs{M}_t$ be the $\bs{A}$-module  
  \begin{displaymath}
    \bs{M}_t = \mc{N}_{D_v}^{-1} \oplus \bigoplus_{\ell = 0}^{k-1}\mc{N}_{D_v}^{k-\ell} \simeq \bs{A}Z_2^{-1}\oplus \bigoplus_{\ell=0}^{k-1} \bs{A}
    Z_1^{\ell}Z_2^{k-\ell}
  \end{displaymath}  
  from the proof of proposition \ref{prop:localfinal}. Objects of the fiber of $\varkappa$ are elements $(\alpha, a_k, \ldots, a_1)$ in $\bs{M}_t$ such that $a_k \in \bs{A}^*$. Write $\underline{Y}_j$ for the set of variables $Y_0 =1, Y_1, \ldots, Y_j$ and let $H_{\ell, j}(\underline{Y})$ be the $j$-coefficient of the formal series in $X$
  \begin{displaymath}
    \frac{1}{(\sum_{q \geq 0} Y_qX)^{\ell}} = \sum_{j \geq 0} H_{\ell, j}(\underline{Y}_j)X^j.
  \end{displaymath}
  The polynomial $H_{\ell, j}(\underline{Y}_j)$ is homogeneous of degree $j$ if $Y_q$ is of degree $q$ and $H_{\ell, 0} = 1$. Let $(\lambda, f(t))$ with $f(t) = t + \alpha_1t^2 + \cdots + \alpha_{k-1}t^k$ be an element in $\Diff_k = \bG_m\ltimes \mc{U}_{k}$. The action of $[x, (\lambda, f(t))]$ on $\varkappa^{-1}(\pi^\star)$ sends $(\alpha, a_k, \ldots, a_1)$ on the $k+1$-tuple $(\alpha^*, a_k^*, \ldots, a_1^*)$ defined by
\begin{equation}
  \label{eq:descriptionaction}
  \alpha^* = \lambda^{-1}\alpha, \quad a_\ell^* = \sum_{r+s = k-\ell} a_{r}\alpha^{s}H_{k-r,s}(\underline{\alpha}_s).  
\end{equation}
Notice that we always have $a_k^* = a_k$, this implies $\Diff_k$ acts indeed on the fiber of $\varkappa$ over $\pi^{\star}$. By construction, away from $\alpha = 0$ this action is the action of $\Diff_k$ on the set of polar parts. It is therefore transitive and has $\bs{\mu}_k$ stabilizers at each point with $\alpha \neq 0$. Now relations \ref{eq:descriptionaction} show clearly that this action of $\Diff_k$ is trivial on elements having $\alpha= 0$.  
\end{proof}
\begin{proposition}
  \label{prop:isoPC}
  In case either $\bs{p} = 0$ or $\bs{p} > k$ there is an isomorphism from the fiber of $\varkappa$ over $(\pi^\star: \mc{X}^\star \rightarrow S, D)$ on $\mc{C}_{k}$ that is compatible with the transition functions on the target and domain. 
\end{proposition}
\begin{proof} 
  The fiber $\mc{C}_{\vk}$ over $\pi^\star$ is the stacky quotient $[C/\bs{\mu}_k]$ where the scheme $C$ is obtained out of $W = \bA_{\sigma_-}\times \left(\bG_m\ltimes \mc{U}_k\right)$ by identifying $[\lambda^{-1}x, g]$ and $[x, \theta_t(\lambda)g(t)]$. By identifying $\Diff_k$ with the orbit of $(1, 0, \ldots, 0) \in [\bA^{k}_{\bs{A}}/\bs{\mu}_k]$ one gets a morphism from the open subset of $W$ given by $\Diff_k$ to $[\bA^{k}_{\bs{A}}/\bs{\mu}_k]$. Explicitly $\theta_t(\lambda)f(t) \in \Diff_k = \bG_m\ltimes \mc{U}_k$ with $f(t) = t + \alpha_1t^2 + \cdot + \alpha_{k-1}t^{k}$ is sent on  
\begin{displaymath}
  \alpha = \lambda^{-1}, \quad a_\ell = H_{k, k-\ell}(\underline{\alpha}_{k-\ell}). 
\end{displaymath}
One can naturally extend this morphism to $W$ by sending $[x, (\lambda, f(t))]$ on the element 
\begin{displaymath}
  \alpha = \lambda^{-1}x, \quad a_\ell = H_{k, k-\ell}(\underline{\alpha}_{k-\ell}). 
\end{displaymath}
It is clear $[\beta^{-1} x, (\lambda, f(t))]$ and $[x, (\beta\lambda, f(t))]$ have the same image and we thus get a morphism from $C$ on $[\bA_{\bs{A}}^k/\bs{\mu}_k]$. Now an element $\zeta \in \bs{\mu}_k$ acting on the left on $(x, \theta_t(\lambda)f(t))$ is sent on 
\begin{displaymath}
\alpha = \lambda^{-1}\zeta^{-1}, \quad  a_{\ell} = H_{k, k-\ell}(\zeta^{k-1}\alpha_{k-1}, \zeta^{k-1} \alpha_{k-1}, \ldots \zeta\alpha_1) = \zeta^{k-\ell} H_{k, k-\ell}(\underline{\alpha}_{k-\ell}).
\end{displaymath}
This shows that we get in fact a morphism $\Sigma_t : [C/\bs{\mu}_k] \rightarrow [\bA^{k}_{\bs{A}}/\bs{\mu}_k]$. To see it is an isomorphism notice that $H_{\ell, j}(\underline{Y}_j) = -\ell Y_j + (\star\star)_{\ell, j}$ where $(\star\star)_{\ell, j}$ is a polynomial $Y_r$ for $r< j$. Since either $\bs{p} = 0$ or $\bs{p} > k$ one can recover the coordinates of $C$ from those on the fiber of $\varkappa$. 

We still need to check how $\Sigma_t$ behaves under a coordinate change $t \mapsto \vartheta t$ for $\vartheta \in \Diff_k$. Let $\Sigma_t$ and $\Sigma_{\vartheta t}$ be the isomorphisms to $\fP_{g, \vk}$ built using $t$ and $\vartheta t$. Recall the transition function from $C_t$ on $C_{\vartheta t}$ is $G_{\vartheta} = \Psi_{\vartheta}R_{\vartheta^{-1}}$. On $\Diff_k$ the morphism $\Psi_\vartheta = \id$ and $R_{\vartheta^{-1}}$ sends an element $g(t) \in \Diff_k$ on $g(t)\vartheta^{-1}(t) = \vartheta^{-1}(g(t))$. Away from the zero section the composition  $\Sigma_{\vartheta t}G_{\vartheta}\Sigma_t^{-1}$ corresponds by construction to the transition map from $[\bA^{k}/\bs{\mu}_k]_t$ to $[\bA^{k}/\bs{\mu_k}]_{\vartheta t}$ given by $t^{-k} \mapsto \vartheta^{-1}(\vartheta t)^{-k}$.  
\end{proof}
\begin{corollary}
If $\bs{p} = 0$ or $\bs{p} > k_i$ for every $i$ the stacks $\fP_{g, \vk}$ and $\mc{C}_{\vk}$ are isomorphic over $\fMm_{g, n}$.  
\end{corollary}
\begin{proof}
  This boils down to checking the existence of local isomorphisms that glue together to define an isomorphism globally from $\fP_{g, \vk}$ to $\fMm_{g, n}$ and this is the content of \ref{prop:isoPC}.  
\end{proof}

\section{Cone structure and projectivization}

Proposition \ref{prop:localfinal} implies $\fP_{g, \vk}$ is \emph{locally} a stack-cone in the sense of \cite{Kainc} over $\fMm_{g, n}$. This is trivially the case for the stack of twisted polar parts. The isomorphism from $\fP_{g, \vk}$ to $\mc{C}_{\vk}$ defines local isomorphisms of cones over $\fMm_{g, n}$. We are going to describe the (global) natural $\bA^1$-structure on $\fP_{g, \vk}$ needed to show it is indeed a stack-cone. It means defining a $0$ section $\bs{0} :\fMm_{g, n} \rightarrow \fP_{g, \vk}$, an $\fMm_{g, n}$-morphism $\bs{\gamma}: \bA^1\times \fP_{g, \vk} \rightarrow \fP_{g, \vk}$ and $2$-commutative diagrams corresponding to the usual compatibilities between $\bs{\gamma}$, $\bs{0}$ and multiplication by $1$ in the scheme theoretic context. In our case $2$-commutativity is obvious and will not be treated. 

It is easier to define $\bs{0}$ and $\bs{\gamma}$ while working with the description of objects of $\fP_{g, \vk}$ given in proposition \ref{lem:retraductionproduitfibre} . We shall however always describe the corresponding constructions without reference to the stable model of the underlying bubbly curve.     

The zero section $\bs{0} : \fMm_{g, n} \rightarrow \fP_{g, \vk}$ sends the stable $n$-marked $S$-curve $(\pi^\star:\mc{X}^\star\rightarrow S, \underline{D}^\star)$ on the stable polar $S$-curve given by 
\begin{equation}
  \label{eq:zerosection1}
  \left(0: \sO_S \rightarrow \mc{N}_{\cD^\star_i}, (0, 1)^{\otimes k_i} \in \Hi^0\left(\frac{\Sym^{k_i}(\sE_0 = \sO_{\mc{X}^\star}\oplus \sO_S)}{\sO_{\mc{X}^\star}\Delta^{k_i}}\right) \right)_{i=1}^n.
\end{equation}
\begin{proposition}
  \label{prop:section0B}
  An $S$-object in the image of $\bs{0}$ is a stable polar curve $(\pi : \mc{X}\rightarrow S, \underline{\varrho})$ where 
    $(\pi : \mc{X} \rightarrow S, \underline{D})$ is obtained out of $\pi^\star$ by gluing $n$ marked projective $S$-lines $(\bP^1_{S, i}, 0, \infty)_i$ along $0$ and $D_i^\star$, and the polar part $\varrho_i$ is the one on $\bP^1_{S,i}$ along $\infty$ that is given by $t_{\infty}^{-k_i} = t_0^{k_i}$.
\end{proposition}
\begin{proof}
The bubbly curve underlying \ref{eq:zerosection1} is the curve $\mc{X} = \Proj_{\mc{X}^\star}(\Sym\sO_{\mc{X}^\star}\oplus \sO_S^{\oplus n})$ whose marking $D_i$ is given by $(\sO_{\mc{X}^\star}\oplus \sO_S^{\oplus n})_{\mid D_i^\star}\twoheadrightarrow \sO_S$. The $\sO_{\mc{X}^\star}$-module structure on $\sO_{\mc{X}^\star}\oplus \sO_S^{\oplus n}$ is the one obtained out of $\sO_{\mc{X}^\star}^{\oplus n+1}$ by moding out the $i+1$ factor by $\sO_{\mc{X}^\star}(-D_i^\star)$. The $i$-th marking is just given by projection on the $i+1$ factor. Consider the $(n+1)$-tuples $Z_1 = (1, 0, \ldots, 0)$ and $Z_{2, i}= (0,\ldots, 0, 1,0, \ldots, 0)$ having $1$ in position $i+1$. We can write 
\begin{equation}
  \label{eq:presentationA}
  \ms{A} = \Sym(\sO_{\mc{X}^\star}\oplus \sO_S^{\oplus n}) = \frac{\sO_{\mc{X}^\star}[Z_1, Z_{2,i} \mid i \in \{1, \ldots, n\}]}{(\sO_{\mc{X}^\star}(-D_i^\star)Z_{2, i}\mid i \in \{1,\ldots, n\})}.
\end{equation}
Away from $\cup_i D_i^\star$ the algebra $\ms{A}$ is isomorphic to $\Sym\sO_{\mc{X}^\star}$ and $\mc{X}$ is isomorphic to $\mc{X}^\star$ on this open subset. Now the restriction of $\ms{A}$ to $D_i^\star$ is the algebra $\sO_S[Z_1,Z_{2, i}]$ and thus $\mc{X}_{\mid D_i^\star} \simeq \Proj_{\sO_S}\sO_S[Z_1,Z_{2, i}] = \bP^1_{S, i}$. The Cartier divisor $D_i^\star$ is given by $Z_1=0$ and that is the zero section of $\bP^1_{S, i}$. This shows $\mc{X}$ is obtained by the announced gluing. Now
\begin{displaymath}
  \frac{\Sym^{k_i}(\sO_{\mc{X}^\star}\oplus \sO_S)}{\sO_{\mc{X}^\star}\Delta^{k_i}} = \frac{\sO_{\mc{X}^\star}[Z_1, Z_{2,i}]_{k_i}}{ \sO_{\mc{X}^\star}Z_1^{k_i} + \sO_{\mc{X}^\star}[Z_1, Z_{2,i}]_{k_i-1}\sO_{\mc{X}^\star}(-D_i^\star)Z_{2, i}}. 
\end{displaymath}
and the polar part $(0,1)^{\otimes k_i} = Z_{2,i}^{k_i}$ corresponds to the local generator $t_\infty^{k_i}$ of $\mc{N}_{D_i^\star}^{\otimes k_i}$.
\end{proof}
Here we define $\bs{\gamma}$. Let $\alpha_i : \ms{R}_i \rightarrow \mc{N}_{D_i^\star}$ be morphisms of $\sO_S$-modules. An element $\lambda \in \sO_S$ defines a morphism   
\begin{equation}
  \label{eq:sflambda}
\begin{tikzpicture}[baseline=(current  bounding  box.center),>=latex]
      \matrix (m) 
      [matrix of math nodes, row sep=4pt, column sep=1em, text
      height=1ex, text depth=0.25ex]  
      {  \lambda[l]: & \sE_{\alpha_i} &
        \phantom{\cdot} & \phantom{\cdot} &
        \sE_{\lambda\alpha_i}
        \\   
        & x\times_{\alpha_i} y  &\phantom{\cdot} & \phantom{\cdot} &
        \lambda x\times_{\lambda\alpha_i} y. \\};  
      \path[->,font=\scriptsize]  
      (m-1-3.west) edge  (m-1-4.east);
      \path[|->, font=\scriptsize]
      (m-2-3.west) edge (m-2-4.east);
    \end{tikzpicture}
\end{equation}
Notice that $\lambda[l]$ is not the morphism induced by functoriality of pull-back of Yoneda extensions. The latter goes in the opposite way and sends $x\times_{\lambda\alpha_i}y$ on $x\times_{\alpha_i}\lambda_iy$. Define $\bs{\gamma}$ as the morphism sending a couple $(\lambda, (\alpha_i:\ms{R}_i\rightarrow \mc{N}_{D_i^\star}, \varrho_{k_i})_i)$ over $(\pi^\star:\mc{X}^\star \rightarrow S, \underline{D}^\star)$ on
\begin{equation}
  \label{eq:A1structure}
\left(\lambda\alpha_i, {\sf m}_{\lambda}(\varrho_{k_i}) \in \Hi^0\left(\frac{\Sym^{k_i}(\sE_{\lambda\alpha_i})}{\sO_{\mc{X}^\star}\Delta^{k_i}}\right)\right)_{i=1}^n,   
\end{equation}
where ${\sf m}_{\lambda}$ stands for the morphism induced by $\lambda[l]$ on the $k_i$ tensor power. It is clear $\bs{\gamma}$ acts along the fibers of $\varkappa$, it is thus an $\fMm_{g,n}$-morphism.
\begin{proposition}
  In the local picture of proposition \ref{prop:localfinal} the $\bA^1$-structure $\bs{\gamma}$ induces a $\bG_m$-action of weights $(1,1, 2, \ldots, k_i-1)$ on the fiber $[\bA_{\bs{A}}^{k_i}/\bs{\mu}_{k_i}]$. 
\end{proposition}
\begin{proof}
Recall locally the fiber of $\varkappa$ over $(\pi^\star : \mc{X}^\star \rightarrow S, \underline{D})$ is the product of the quotient stack $[\bA_{\bs{A}}^{k_i}/\bs{\mu}_{k_i}]$ where $\bs{\mu}_{k_i}$ acts with weights $(1,1,2,\ldots, k_i-1)$. Keep the notation of proposition \ref{prop:localfinal} and drop the $i$ indices. Let 
\begin{displaymath}
\left(\alpha: \mc{N}_{D_\alpha} \rightarrow \bs{A}, \varrho \in \frac{\Sym^k(\bs{E}_\alpha)}{\sO Z_1^{k}}\right)
\end{displaymath}
be an $\bs{A}$-object in $\fP_{g, \vk}$ over $\pi^\star$. There is $x^{j, lj} \in \sO_{\mc{X}^\star}(D^\star)$ and $y^{j, l_j} \in \ms{R}$ for $j \in \mr{J}$ and $l_j \in \{1, \ldots, k\}$ such that 
\begin{equation}
  \label{eq:calcul1}
\varrho = \sum_{j\in \mr{J}} \bigotimes_{l_j=1}^k (x^{j, l_j} \times_{\alpha}
  y^{j, l_j}). 
\end{equation} 
Let $t$ be an equation of $D^\star$. Writing $x^{j, l_j} = \alpha y^{j, l_j} + tz^{j, l_j}$ and using the local presentation of $\bs{E}_\alpha$ lemma \ref{lem:localEalpha} one gets
\begin{equation}
 \label{eq:calcul5}
\varrho  = \sum_{\ell =0}^{k-1} \sum_{j \in \mr{J}}\big(\sum_{\substack{\mr{I} \subset \{1, \ldots, k\}\\ |\mr{I}|=
    \ell}} \prod_{\ell_j \in \mr{I}} z^{j, \ell_j} \prod_{\bar{\ell}_j \in \complement\mr{I}} y^{j, \bar{\ell}_j} \big) \,\, Z_1^\ell Z_2^{k-\ell} = \sum_{\ell=0}^{k-1} a_{k-\ell} Z_1^\ell Z_2^{k-\ell}.
\end{equation}
By definition ${\sf m}_{\lambda}(\varrho) = \sum_{j\in \mr{J}} \bigotimes_{l_j=1}^k (\lambda z^{j, l_j} \times_{\lambda\alpha}  y^{j, l_j})$. In terms of $Z_1$ and $Z_2$ one gets
\begin{equation}
  \label{eq:actionlambdaloc}
  {\sf m}_\lambda(\varrho)  = \sum_{\ell =0}^{k-1} \sum_{j \in \mr{J}}\big(\sum_{\substack{\mr{I} \subset \{1, \ldots, k\}\\ |\mr{I}|= \ell}} \prod_{\ell_j \in \mr{I}} \lambda z^{j, \ell_j} \prod_{\bar{\ell}_j\in \complement\mr{I}} y^{j, \bar{\ell}_j} \big)Z_1^\ell(\lambda Z_2)^{k-\ell} 
 = \sum_{\ell=0}^{k-1} \lambda^{\ell}a_{k-\ell}
Z_1^\ell(\lambda Z_2)^{k-\ell}.
\end{equation}
$\lambda$ acts trivially on $a_k$ and with the expected weights on the coordinates of $[\bA_{\bs{A}}^k/\bs{\mu}_k]$.
\end{proof}
\begin{corollary}
\label{cor:proprtePPgk}
The quotient stack $\bP\fP_{g, \vk} = [\fP_{g, \vk} \setminus \{\bs{0}\}/\bG_m]$ is proper.
\end{corollary}
\begin{proof}
Locally $\fP_{g, \vk}\rightarrow \fMm_{g, n}$ is a finite product of weighted projective stacks and $\fP_{g, \vk} \rightarrow \fMm_{g, n}$ is thus proper. Since $\fMm_{g, n}$ is proper (\cite{Knudsen2}) this is the case as well for $\bP\fP_{g, \vk}$. 
\end{proof}
It is instructive to see that $\bs{\gamma}$ acts on the set of polar parts over the underlying bubbly curve. An element $\lambda \in \sO_S^*$ defines an isomorphism $\lambda[r] : \sE_{\lambda\alpha_i} \rightarrow \sE_{\alpha_i}$ by functoriality of the pull-back of Yoneda extensions. It sends $x\times_{\lambda\alpha_i} y$ on $x \times_{\alpha_i} \lambda y$. We get this way an isomorphism 
\begin{displaymath}
\left(\lambda\alpha_i, {\sf m}_\lambda(\varrho_i) \in \Hi^0\left(\frac{\Sym^{k_i}\sE_{\lambda\alpha_i}}{\sO\Delta^{k_i}}\right)\right)_{i}\simeq \left(\alpha_i, \hat{\lambda}{\sf m}_\lambda(\varrho_i) \in \Hi^0\left(\frac{\Sym^{k_i}\sE_{\lambda\alpha_i}}{\sO\Delta^{k_i}}\right)\right)_{i}
\end{displaymath}
 where $\hat{\lambda}$ is the isomorphism induced by $\lambda[r]$ between the symmetric algebras $\Sym\sE_{\lambda\alpha_i}$ and $\Sym\sE_{\alpha_i}$. Now $\hat{\lambda}{\sf m}_{\lambda}$ is the map $x\times_{\alpha_i} y \mapsto \lambda (x\times_{\alpha_i} y)$. This means we have an isomorphism  
\begin{displaymath}
\left(\lambda\alpha_i, {\sf m}_\lambda(\varrho_i) \in \Hi^0\left(\frac{\Sym^{k_i}\sE_{\lambda\alpha_i}}{\sO\Delta^{k_i}}\right)\right)_{i} \simeq \left(\alpha_i,
  \lambda^{k_i}\varrho_{k_i} \in \Hi^0\left(\frac{\Sym^{k_i}\sE_{\alpha_i}}{\sO\Delta^{k_i}}\right)\right)_{i}.
\end{displaymath} 
This can be understood by saying that the $\bG_m$-action on $\fP_{g, \vk} \rightarrow \fMm_{g, n}$ sends a stable polar curve $(\pi : \mc{X} \rightarrow S, \varrho_i)$ on $(\pi:\mc{X} \rightarrow S, \lambda^{k_i}\varrho_i)$. The difference between expression \ref{eq:A1structure} and this one is the fact the former preserves the trivializations of $\sL_i^{-k_i}$ while the latter preserves the underlying bubbly curve. 

Unfortunately, this $\bG_m$-action is not the action on polar parts that comes from the $\sO_{S}$-module structure and that sends $(\alpha_i, \varrho_i)$ on $(\alpha_i, \lambda \varrho_i)$. The good news is that both $\bG_m$-actions are equivalent on the complement of the zero section of $\varkappa : \fP_{g,\vk} \rightarrow \fMm_{g, n}$. Let $\nu[\vk]$ be the morphism sending a polar curve $(\alpha_i, \varrho_i)$ on $(\alpha_i, \smash{\varrho_i^{k_i}})$ where $\smash{\varrho_i^{k_i}}$ is the $k_i$-th tensor power of $\varrho_i$ that lies in 
\begin{equation}
\label{eq:tensorpolarparts}
\left(\frac{\sO_{\mc{X}_{\alpha_i}}(k_iD_i)}{\sO_{\mc{X}_{\alpha_i}}}\right)^{\otimes k_i} = \frac{\sO_{\mc{X}_{\alpha_i}}(k_i^2D_{\alpha_i})}{\sO_{\mc{X}_{\alpha_i}}(k_i(k_i-1)D_{\alpha_i})}. 
\end{equation}
By tensoring the previous sheaf with $\sO_{\mc{X}_{\alpha_i}}(-k_i(k_i-1)D_{\alpha_i})$ one gets back the sheaf of polar parts and can thus locally identify $\smash{\varrho_i^{k_i}}$ with a well defined polar part. It is straightforward to check that the stack of polar curves where one replaces the sheaf of polar parts by \ref{eq:tensorpolarparts} is isomorphic to $\fP_{g, \vk}$ over $\fMm_{g, n}$. We therefore look at $\nu[\vk]$ as a morphism having target $\fP_{g, \vk}$. Let $G_{i, j}(Y_{k_i}, \ldots, Y_1)$ be the polynomial in $Y_\ell$ which is the coefficient of $X^j$ in 
\begin{displaymath}
(Y_{k_i}X^{k_i} + Y_{k_i-1}X^{k_i-1} + \cdots + Y_1X^1)^{k_i}.
\end{displaymath}
The polynomial $G_{i,j}$ is homogeneous of degree $k_i$ in $Y_1, \ldots, Y_{k_i}$ and $G_{i, k_i} = Y_{k_i}^{k_i}$. Locally $\nu[\vk]$ is defined by sending an $\bs{A}$-point $(\alpha_i, a_{k_i}, \ldots, a_1)$ on 
\begin{displaymath}
(\alpha_i, G_{i, k_i}(\underline{a}), G_{i, k_i-1}(\underline{a}), \ldots, G_{i, 1}(\underline{a}))
\end{displaymath}
where $\underline{a}$ is short for $(a_{k_i}, \ldots, a_1)$. It is not hard to see that this map is indeed defined on $\prod_i\bP_{\bs{A}}(-1, k_i, k_i-1, \ldots, 1)$ where $a_{k_i} \neq 0$ and it is clear it is $\bG_m$-equivariant. Restrict $\nu[\vk]$ to the complement of the zero section $\alpha_i= a_{k_i-1} = \cdots = a_1=0$. Now $G_{i, j}(\underline{Y}) = b_{j}$ is a triangular system of equations that has a unique solution if we fix a $k_i$-th root of unity of $b_{k_i}$. Any other solution is obtained by multiplying by a $k_i$-th root of unity. This says $\nu[\vk]$ defines an isomorphism from $\prod_i[\bA_{\bs{A}}^{k_i}/\bs{\mu}_{k_i}]\setminus \{\bs{0}\}$ on itself. This is enough to claim that the projectivizations of $\fP_{g, \vk}$ for both $\bG_m$-actions are isomorphic.    

\section{Residues and the ELSV compactification}

\subsection{Definition of the ELSV cone}

Let $(\pi: \mc{X}\rightarrow S, \underline{\varrho})$ be a stable polar curve in $\fP_{g, \vk}$ and consider the exact sequence 
\begin{equation}
\label{equationabla}
\begin{tikzpicture}[baseline=(current  bounding  box.center),>=latex, descr/.style={fill=white,inner sep=2.5pt}]
    \matrix (m) 
    [matrix of math nodes, row sep=2.5em, column sep=2em, text
    height=1ex, text depth=0.25ex] 
    {   0  & \sO_S & \pi_*\sO_{\mc{X}}(\sum_{i=1}^n k_iD_i) &
      \sP_{\mc{X}, \vk} \\
      & \Ri^1\pi_*\sO_{\mc{X}} & \Ri^1\pi_*\sO_{\mc{X}}(\sum_{i=1}^n k_iD_i) & 0. \\};
    \path[->,font=\scriptsize]  
    (m-1-1) edge (m-1-2)
    (m-1-2) edge (m-1-3)
    (m-1-3) edge (m-1-4)
    (m-2-2) edge (m-2-3)
    (m-2-3) edge (m-2-4)
    (m-1-4) edge[out=-10, in=170] node[descr] {$\nabla_{\mc{X}}(\vk)$} (m-2-2);
  \end{tikzpicture}
\end{equation}
The theory of base change in cohomology for proper flat morphisms says $\sP_{\mc{X}, \vk}$, $\Ri^1\pi_*\sO_{\mc{X}}$ and $\Ri^1\pi_*\sO_{\mc{X}}(\sum_{i=1}^nk_iD_i)$ are locally free sheaves commuting to any base change. As a result $\nabla_{\mc{X}}(\vk)$ commutes to any base change. To each $S$-object $(\pi:\mc{X}\rightarrow S, \underline{\varrho})$ we can thus attach a global section $\nabla_{\mc{X}}(\vk)(\underline{\varrho})$ in $\Ri^1\pi_*\sO_{\mc{X}}$. We get in this way a section $\nabla(\vk)$ of the vector bundle over $\fP_{g, \vk}$ defined by $\varkappa^*\Ri^{1}\pi_*\sO_{\fMm_{g, n}}$. By Grothendieck-Serre duality we have that \begin{displaymath}
\Ri^1\pi_*\sO_{\mc{X}} \simeq (\pi_*\omega_{\mc{X}})^\vee. 
\end{displaymath}
Following \cite{HartRes} and \cite{BrianConrad} this isomorphism commutes to base change. This shows $\nabla(\vk)$ does in fact define a section $\nabla(\vk)^\vee$ of $\varkappa^*\mbb{E}_{g, n}^\vee$ where $\mbb{E}_{g, n}$ is the Hodge bundle on $\fMm_{g, n}$. Recall fibers of $\mbb{E}_{g, n}$ over a marked curve are global sections of the dualizing sheaf, see \cite{ArbaCorGriff} for further details. 
\begin{definition}
  \label{defn:Z}
  The ELSV cone $\fZ_{g, \vk}$ is the zero locus of $\nabla(\vk)^{\vee}$. 
\end{definition}
In the case of a polar curve $(X, \underline{\varrho})$ with smooth $X$ the section $\nabla(\vk)^{\vee}(\underline{\varrho})$ is just the residue linear form 
\begin{displaymath}
  \kappa \longmapsto \sum_{i=1}^n \Res_{p_i}(\kappa \varrho_i)
\end{displaymath}
where $\kappa$ is a regular $1$-form on $X$. The characterization of meromorphic functions by their residues says $(X, \underline{\varrho})$ is in $\fZ_{g, \vk}$ if and only if $\underline{\varrho}$ is the collection of polar parts of a meromorphic function on $X$ at the marked points. This theorem extends to the case of any polar curve by replacing regular $1$-forms with sections of the dualizing sheaf. Since the marked points are in the smooth locus residues at the marked points are the usual ones. See \cite{ELSV} for the above formulation of this classical result. It can be deduced of the long exact sequence \ref{equationabla} and the expression of Serre duality which is explicit in this case (see for instance \cite[10.2]{ArbaCorGriff}). 

The previous analysis holds locally over the base. Let $(\pi: \mc{X} \rightarrow S, \underline{\varrho})$ be a polar curve in $\fZ_{g, \vk}$. Take an affine open cover $\{U_{\ell} \rightarrow S\}_\ell$ of $S$ and let $\mc{X}_\ell$ be the pullback of $\mc{X}$ along $U_\ell \rightarrow S$. Assume $U_\ell = \Spec(\bs{A}_\ell)$ with $\bs{A}_\ell$ a local $\bbk$-algebra. Pulling back the long exact sequence \ref{equationabla} along $U_\ell \rightarrow S$ we get the exact sequence 
\begin{equation}
\label{equationablalocal}
\begin{tikzpicture}[baseline=(current  bounding  box.center),>=latex]
    \matrix (m) 
    [matrix of math nodes, row sep=2.5em, column sep=2em, text height=1ex, text depth=0.25ex] 
    {   0  & \bs{A}_\ell & \Hi^0\big(\sO_{\mc{X}_\ell}(\sum_{i=1}^n
      k_iD_{\ell, i})\big) & \Hi^0(\sP_{\mc{X}_{\ell}, \vk}) &
      \Hi^1(\sO_{\mc{X}_{\ell}}). \\}; 
    \path[->,font=\scriptsize]  
    (m-1-1) edge (m-1-2)
    (m-1-2) edge (m-1-3)
    (m-1-3) edge (m-1-4)
    (m-1-4) edge (m-1-5);
  \end{tikzpicture}  
\end{equation}  
By assumption the image of $\underline{\varrho}_\ell$ by $\nabla(\vk)$ is zero in $\Hi^1(\sO_{\mc{X}_\ell})$. We can therefore lift $\underline{\varrho}_{\ell}$ to a section in $\Hi^0(\sO_{\mc{X}_\ell}(\sum_{i=1}^nk_iD_{\ell, i}))$. This means we have local lifts of $\underline{\varrho}$ to $\pi_*\sO_{\mc{X}}(\sum_{i=1}^nk_iD_i)$. Any lift of $\underline{\varrho}_\ell$ is defined up to the addition of a constant in $\bs{A}_\ell$. A polar curve in $\fZ_{g, \vk}$ is given locally on the base by meromorphic functions (up to the addition of a constant) on the underlying bubbly curve having the expected pole order along the markings. 

\subsection{Relation to stable maps}  
A stable $S$-map of degree $d$ from a curve of genus $g$ on $\bP^1_S$ is a degree $d$ projective $S$-morphism $F : \mc{X} \rightarrow \bP^1_S$ from a prestable $n$-marked $S$-curve of genus $g$ into $\bP^1$ such that automorphism groups of fibers $(\mc{X}_s, F_s)$ are finite. It is equivalent to the fact each smooth rational contracted component of $\mc{X}_s$ has at least $3$ nodal or marked points. Stable maps define a proper Deligne--Mumford stack $\fMm_{g,n}(\bP^1, d)$ (\cite{BehManin}). Automorphisms of stable maps are automorphisms of the marked domain curves commuting to the maps. For further details see  \cite{FultPandh} and \cite{BehManin}. 

The stack of stable maps comes with $n$ evaluation maps $\bs{ev}_1, \ldots, \bs{ev}_n$ : given a stable map $F : \mc{X} \rightarrow \bP^1_S$ from the marked curve $(\pi : \mc{X} \rightarrow S, \underline{D})$ the evaluation map $\bs{ev}_i$ sends $D_i$ on the $S$-point of $\bP^1_S$ given by $\bs{ev}_i\circ \sigma_i$ where $\sigma_i$ is the section defining $D_i$. We write $\fMm_{g,n}(\bP^1, d, \infty)$ for the closed substack of $\fMm_{g,n}(\bP^1, d)$ corresponding to the intersection locus of $\bs{ev}_i^{-1}\infty$.
\begin{definition}
    Denote $\fMm_{g,n}(\bP^1, \vk)$ the substack of $\fMm_{g, n}(\bP^1, d, \infty)$ given for $S \in \Sch$ by $S$-maps $F: \mc{X} \rightarrow \bP_S^1$ satisfying scheme theoretically
  \begin{equation}
    \label{conditionprofile}
    \phi^{-1}\infty = \sum_{i=1}^nk_iD_i.
  \end{equation}
\end{definition}
The stack $\fMm_{g, n}(\bP^1, \vk)$ comes with a natural action of $\Aut(\bP^1, \infty)$. An automorphism $\alpha \in \Aut(\bP^1, \infty)$ sends an $S$-map $F$ on $\alpha \circ F$. We identify $\Aut(\bP^1, \infty)$ with $\bG_m\ltimes \bG_a$ where $\bG_a$ fixes $\infty$ and $\bG_m$ both $0$ and $\infty$.
\begin{lemma}
  $\fMm_{g, n}(\bP^1, \vk)$ is a locally closed (thus algebraic) substack of $\fMm_{g, n}(\bP^1, d, \infty)$.  
\end{lemma}
\begin{proof}
  Let $F: \mc{X} \rightarrow \bP^1_S$ an $S$-object in $\fMm_{g, n}(\bP^1, d, \infty)$. The condition $\sum k_i D_i \subset \phi^{-1}(\infty)$ is closed on the base. Indeed, $F$ is given by an onto map $(t_0,t_\infty): \sO_\mc{X}^{\oplus 2} \twoheadrightarrow F^*\sO_{\bP_S^1}(1)$ for which the previous condition is expressed by $t_{\infty \mid \sum_i k_iD_i} = 0$. Now following the Zariski main theorem the conditions for $\phi^{-1}(\infty) \rightarrow S$ to be quasi-finite and (thus) finite are open on $S$. We can assume that $\phi^{-1}(\infty) \rightarrow S$ is finite and that $\sum k_i D_i \subset \phi^{-1}(\infty)$. Under these assumptions $\phi^{-1}(\infty) = \sum_i k_iD_i$ is a closed condition on $S$. 
\end{proof}
Let $x_0$ and $x_\infty$ be generators of $\sO_{\bP_S^1}(1)$ at the $0$ and $\infty$ $S$-points. A map $F : \mc{X} \rightarrow \bP^1_S$ in $\fMm_{g, n}(\bP^1, \vk)$ is given by an onto morphism 
\begin{displaymath}
\begin{tikzpicture}[baseline=(current  bounding  box.center),>=latex]
    \matrix (m) 
    [matrix of math nodes, row sep=2.5em, column sep=2em, text height=1ex, text depth=0.25ex] 
    { \sO_{\ms{X}}^{\oplus 2} & F^*\sO_{\bP_S^1}(1)\\};  
    \path[->>,font=\scriptsize]  
    (m-1-1) edge (m-1-2);
  \end{tikzpicture}
\end{displaymath}
defined by $s_0 =F^*x_0$ and $s_\infty = F^*x_\infty$. Away from $F^{-1}\infty = \sum_ik_iD_i$ the section $s_\infty$ is invertible and $s_0s_\infty^{-1}$ defines a global section on $\mc{X} \setminus \cup_iD_i$. Since $\sum_ik_iD_i$ is a relative Cartier divisor $s_0s_\infty^{-1}$ extends to a relative meromorphic section $\zeta_{F}$ on $\mc{X}$. By assumption $\zeta_F \in \sO_{\mc{X}}(\sum_i k_i D_i)$ and its image $\overline{\zeta}_{\bs{F}}$ in $\sP_{\mc{X}, \vk}$ is a well defined polar part of order $\vk$. This procedure defines a forgetful map from $\fMm_{g, n}(\bP^1, \vk)$ on $\fZ_{g, \vk}$. This is the case because $(\mc{X}, \underline{D})$ is bubbly: any unstable component of a fiber has to dominate $\bP^1$ but this is the case if and only if it has a marked point.   
\begin{proposition}
  The map $\fMm_{g, n}(\bP^1, \vk) \rightarrow \fZ_{g,\vk}$ is a $\bG_a$-torsor, thus $[\fMm_{g,n}(\bP^1, \vk)/\bG_a] \simeq \fZ_{g, \vk}$.
\end{proposition}
\begin{proof}
This is a local statement. Assume $S = \Spec(\bs{A})$ for a local $\bbk$-algebra $\bs{A}$ and that $(\pi : \mc{X} \rightarrow S, \underline{\varrho})$ is a polar curve in $\fZ_{g, \vk}$. By assumption $\underline{\varrho}$ has a lift $\zeta \in \Hi^0(\sum_ik_iD_i)$. Looking back at the long exact sequence \ref{equationablalocal} any such lift is of the form $\zeta + a\bs{1}$ for $a \in \bs{A}$ and $\bs{1}$ the canonical section $\sO_{\mc{X}} \rightarrow \sO_{\mc{X}}(\sum_ik_iD_i)$. Given $\zeta$ one defines an $S$-map $F_{\zeta} : \mc{X} \rightarrow \bP_S^1$ by looking at the local generators $\{\zeta, \bs{1}\}$ of $\sO_{\mc{X}}(\sum_ik_iD_i)$. These generate $\sO_{\mc{X}}(\sum_ik_iD_i)$ because locally $\zeta = f_i^{-k_i}$ for an equation $f_i$ of $D_i$. In particular $F_\zeta$ satisfies \ref{conditionprofile}. Away from $\cup_iD_i$ the map $F_{\zeta}$ is defined by $\zeta$ and by $\zeta^{-1}$ at the neighborhood of $\cup_iD_i$. The map $F_\zeta$ is stable because $\zeta$ is never colinear to $\bs{1}$ at the neighborhood of a marked point and thus $F_{\zeta}$ does never contract a bubble. It is clear the action of $a \in \bG_a(\bs{A})$ on $F_{\zeta}$ corresponds to the action of $a$ sending $\zeta$ on $\zeta + a\bs{1}$. To prove the statement we only need to notice $\bs{A}$ acts freely on lifts of $\underline{\varrho}$, but this is always the case for a lift is never colinear to $\bs{1}$ at the marked points. 
\end{proof}

\subsection{The ELSV compactification} 

Let $\fM_{g, n}(\bP^1, \vk)$ be the open substack of $\fMm_{g, n}(\bP^1,\vk)$ corresponding to maps having smooth domains. An $S$-map $F :\mc{X} \rightarrow \bP_S^1$ in $\fM_{g, n}(\bP^1, \vk)$ defines in particular an $S$-object in $\cH_{g, \vk}^\bullet$. Conversely, for an object in $\cH_{g, \vk}^\bullet$ one gets a map in $\fMm_{g, n}(\bP^1, \vk)$ if given an isomorphism from the target on $(\bP^1, \infty)$. This is always possible locally on the base. The map from $\fM_{g, n}(\bP^1, \vk)$ on $\cH_{g, \vk}^\bullet$ built previously is, however, not an isomorphism because objects of $\cH_{g, \vk}^\bullet$ have more automorphisms than those of $\fM_{g, n}(\bP^1,\vk)$. 
\begin{proposition}
  The quotient stack $[\fM_{g, n}(\bP^1, \vk)/\Aut(\bP^1, \infty)]$ is isomorphic to $\cH_{g, \vk}^\bullet$. 
\end{proposition}
\begin{proof}
  The forgetful map $\fM_{g, n}(\bP^1, \vk) \rightarrow \cH_{g, \vk}^\bullet$ is $\Aut(\bP^1, \infty)$-equivariant and gives a map from the expected quotient to $\cH_{g, \vk}^\bullet$. Let's build an inverse $\Sigma$. Let $\phi :\mc{X} \rightarrow \mc{Y}$ where $\mc{Y}$ is marked by $E_\infty$ be an $S$-object in $\cH_{g, \vk}^\bullet$. The group scheme $G = \Aut_S((\mc{Y}, E_\infty),(\bP_S^1, \infty))$ is an $\Aut(\bP^1,\infty)$-torsor over $S$. Write $\Sigma(\phi)$ for the the object of $[\fM_{g, n}(\bP^1, \vk)/\Aut(\bP^1, \infty)]$ given by the $S$-torsor $G \rightarrow S$ together with the morphism to $\fM_{g, n}(\bP^1, \vk)$ defined for every $T$-point $\beta$ of $G$ by $\beta \circ \delta^*(\phi)$. An isomorphism $(\alpha, \beta)$ from $\phi$ to $\psi : \mc{Z} \rightarrow \mc{T}$ in $\cH_{g, \vk}^\bullet$ defines an isomorphism from $G$ on $\Aut_S((\mc{T}, E_\infty),(\bP_S^1, \infty))$ coming from the composition (on the left) of a section of $G$ with $\beta^{-1}$. The isomorphism from $\Sigma(\phi)$ on $\Sigma(\psi)$ is thus simply defined by $\alpha$. It is clear $\Sigma$ is the needed inverse. 
\end{proof}
Since $\bG_a$ is normal in $\Aut(\bP^1, \infty)$, following \cite[Remark 2.4]{MRomagnyG} we have that 
\begin{displaymath}
 \cH_{g, \vk}^\bullet = \left[[ \fM_{g,n}(\bP^1, \vk)\left\slash\bG_a\right. ]\left\slash \bG_m\right.\right] = \left[\fZ_{g, \vk}^{\circ} \left\slash \bG_m\right. \right]
\end{displaymath}
where $\fZ_{g, \vk}^{\circ}$ is the open locus in $\fZ_{g, \vk}$ corresponding to underlying smooth curves and the action of $\lambda \in \bG_m(S)$ on $(\pi : \mc{X} \rightarrow S, \underline{\varrho})$ is the one multiplying each polar part by $\lambda$. We've already seen that this action corresponds to the $\bG_m$-action on fibers of $\varkappa$ away from the zero section. We get therefore that $\cH_{g, \vk}^\bullet \subset \bP\fZ_{g, \vk}$. In order to compactify $\cH_{g, \vk}^\bullet$ or $\cH_{g, \vk}$ it would be natural to take their closure in $\bP\fZ_{g, \vk}$. The closure of a substack $\mf{E}$ in a stack $\mf{C}$ (see \cite[050A]{stacks-project}) is the unique reduced stack $\overline{\mf{E}}$ whose underlying set of points $|\overline{\mf{E}}|$ (in the sense of \cite[5]{Laumon}) is the closure in $|\mf{C}|$ of the set of points of $\mf{E}$. Details on the closure of a stack can be found at \cite[050A]{stacks-project}. The closures $\cHh_{g, \vk}^\bullet$ and $\cHh_{g, \vk}$ respectively compactify $\cH_{g, \vk}^\bullet$ and $\cH_{g, \vk}$ if both are reduced stacks. This is the case if the characteristic of the base field $\bs{p}$ is $0$ or does not divide any of the $k_i$s. In this case both $\cH_{g, \vk}^\bullet$ and $\cH_{g, \vk}$ are smooth. This fact results from work of \cite{HurHM} or by standard deformation theoretic arguments as in \cite{Rav01}.
\begin{definition}
  In case the characteristic of the base field $\bs{p}=0$ or $\bs{p}\nmid k_i$ for every $i$ we define the ELSV compactification $\cHh_{g, \vk}$ (resp. $\cHh_{g, \vk}^\bullet$) as the (stacky) closure of $\cH_{g, \vk}$ (resp. $\cH_{g, \vk}^\bullet$) in $\bP\fZ_{g, \vk}$. 
\end{definition}  
Following \cite[3.6]{ELSV} the ELSV compactification is nearly never dense in $\bP\fZ_{g, \vk}$. It is thus natural to try to have a better understanding of its boundary points.

\section{Preliminaries on admissible covers}

Details about the stack of admissible covers can be found in \cite{HurHM} and \cite{HurBR}. We fix here the type of admissible covers we're interested in. 
\begin{tam}
When working with admissible covers we implicitly assume the characteristic of the base field $\bs{p}$ is $0$ or greater than $d$. This ensures that the stacks of admissible covers make sense and are smooth on the locus of underlying smooth curves.
\end{tam}
Let $\bs{\epsilon} =\{\epsilon_j\}_{j=1}^{\bar{r}}$ be a set of partitions of $d$. Write $\fHh_{g,\vk}^{\bs{\epsilon}}$ for the stack of admissible covers whose $S$-points are finite degree $d$ onto $S$-maps $\phi$ from a stable $g$ curve $(\pi : \mc{X} \rightarrow S,  \underline{D}, \{\underline{\bs{D}}_j\}_{j=1}^{\bar{r}})$ on a genus $0$ stable curve $(\eta : \mc{Y}\rightarrow S , E_{\infty}, \{E_j\}_{j=1}^{\bar{r}})$ such that   
\begin{enumerate}
\item for a nodal point $x \in \mc{X}_s$ such that $\phi(x) = y$ one can find \'etale locally $m_x$ in $\N^*$, $t \in \ms{M}_x$ and $T \in \ms{M}_y$ such that
  \begin{displaymath}
    \sO_x \simeq \frac{\sO_s\llbracket u, v\rrbracket}{(uv-t)} \quad \textrm{and} \quad \sO_y \simeq \frac{\sO_s\llbracket U, V\rrbracket}{(UV-T)}  
  \end{displaymath}
and for which $\phi$ is locally defined by $U=u^{m_x}$, $V=v^{m_x}$ and $T = t^{m_x}$. 
\item $\phi$ is \'etale on the smooth locus of $\pi$ away from the marked $S$-points
\item $\phi$ satisfies the scheme theoretic equalities 
  \begin{displaymath}
    \phi^{-1}E_j = \sum_{\ell_j} \epsilon_j(\ell_j)\bs{D}_{j, \ell_j} \quad \textrm{and} \quad \phi^{-1}E_\infty = \sum_{i=1}^n k_iD_i.
  \end{displaymath}
\end{enumerate}
Automorphisms of $\phi$ are automorphisms $(\alpha, \beta)$ of underlying marked curves, possibly exchanging markings $\bs{D}_{j, \ell_j}$ for fixed $j$, such that $\phi \circ \alpha = \beta\circ \phi$. The substack of admissible covers having smooth underlying curves is written $\fH_{g, \vk}^{\bs{\epsilon}}$. It is an open dense substack of the proper stack $\fHh_{g, \vk}^{\bs{\epsilon}}$. In general $\fHh_{g, \vk}^{\bs{\epsilon}}$ is singular, its normalization is however a smooth stack which has been studied in \cite{HurMochi}, \cite{HurACV} and \cite{HurBR}. The stack of admissible covers in the simply ramified case away from $E_\infty$ is written $\fHh_{g, \vk}$. A result of \cite{Wajnryb} states that $\fHh_{g, \vk}$ is irreducible. We denote $\fHh_{g, \vk}^\bullet$ the finite disjoint union of admissible covers over all profiles $\bs{\epsilon}$ for fixed $d$. Each stack of admissible cover appears as a correspondence between two sets of stable curves. By this we mean it is enclosed in a diagram of forgetful maps
\begin{equation}
\label{corsp:H}
\begin{tikzpicture}[baseline=(current  bounding  box.center),>=latex] 
      \matrix (m) 
      [matrix of math nodes, row sep=2.5em, column sep=2.5em, text
      height=1.5ex, text depth=0.25ex]  
      { \fMm_{g, n, \dots} & \fHh_{g, \vk}^{\bs{\epsilon}} & \fMm_{0, \bar{r}+1}. \\};    
      \path[->,font=\scriptsize]  
      (m-1-2) edge (m-1-3)
      (m-1-2) edge (m-1-1);
  \end{tikzpicture}
\end{equation} 
The left hand-side map is the forgetful one keeping only the stable marked domain curve. The right-hand side does keep the target marked stable curve, it is to be understood as an extension of the standard branch divisor of a cover to the case of admissible covers. This extension can also be recovered independently. A branch point of an admissible cover $\phi : X \rightarrow Y$ is a branch point of the induced map on the normalizations of $X$ and $Y$ that lands in the smooth locus of $Y$. Since points in the fibers of branch points of $\phi$ are smooth one can define the branch divisor in the usual way. In general the branch morphism $\Br(\phi)$ of an admissible $S$-cover $\phi : \mc{X} \rightarrow \mc{Y}$ is the $\Div$ of the perfect torsion complex in degrees $-1$ and $0$ given by $\Ri\phi_*[\phi^*\omega_{\mc{Y}/S} \rightarrow \omega_{\mc{X}/S}]$. The morphism $\mr{d}^\omega \phi : \phi^*\omega_{\mc{Y}/S} \rightarrow \omega_{\mc{X}/S}$ is the standard one if $\mc{Y}$ and $\mc{X}$ are smooth over $S$. It does extend to the case of admissible covers \cite{HurBR} and is an isomorphism at the nodal points. The branch divisor is thus supported in the relative smooth locus of $\mc{Y}$. If $\phi$ is an $S$-object in $\fHh_{g, \vk}^{\bs{\epsilon}}$ we have that 
\begin{displaymath}
  \Br(\phi) = (d-n) E_\infty + \sum_{j=1}^{\bar{r}} \big(\sum_{\ell_j}(\epsilon_j(\ell_i)-1)\big) E_j.
\end{displaymath}

Keeping only $E_\infty$ and $\underline{D}$ gives a morphism $\fH_{g, \vk}^\bullet \rightarrow \cH_{g, \vk}^\bullet$. Starting with a point in $\cH_{g, \vk}^\bullet$ one only needs to specify an ordering on branch points away from $\infty$ to get back an admissible cover. Following this idea it is straightforward to see for instance that $\fH_{g, \vk} \rightarrow \cH_{g, \vk}$ is an $\mf{S}_r$-torsor and thus $\cH_{g, \vk} = [\fH_{g, \vk}\slash \mf{S}_r]$. More generally we have an embedding of $[\fH_{g, \vk}^{\bs{\epsilon}}/\mf{S}_{\bar{r}}]$ in $\cH_{g, \vk}^\bullet$. In particular $\fH_{g, \vk}^\bullet \rightarrow \cH_{g, \vk}^\bullet$ is onto and finite.  

\section{Extension of the $\mc{L}\!\mc{L}$ morphism}

\begin{tam}
\label{tam}
From this point on we assume that the characteristic of the base $\bs{p} = 0$ or doesn't divide $r = 2g-2+d+n$ and none of the $k_i$s. The assumption $\bs{p} \nmid r$ is needed to define the $\mc{L}\!\mc{L}$ morphism. 
\end{tam}
Given a  map $F : X \rightarrow \bP^1$ for smooth $X$ the $\mc{L}\!\mc{L}$ morphism introduced in \cite{ELSV} sends $F$ on the divisor supported on finite critical points of $\mr{d}F$. These critical points are branch points of $F$ in $\bP^1\setminus \infty$ and the $\mc{L}\!\mc{L}$ divisor is equal to the branch divisor on this set. In fact this $\mc{L}\!\mc{L}$ divisor is the branch divisor away from $\infty$. 

B.~Fantechi and R.~Pandharipande showed in \cite{BranchFP} that one can extend the definition of the branch divisor to stable maps having smooth target. The key point is to notice that for an $S$-map $F : \mc{X} \rightarrow \mc{Y}$ with $\mc{Y}/S$ smooth, the natural morphism $ \mr{d}^\omega F : F^*\Omega_{\mc{Y}/S} = F^*\omega_{\mc{Y}/S} \rightarrow \Omega_{\mc{X}/S} \rightarrow \omega_{\mc{X}/S}$ defines a perfect torsion complex in degrees $-1$ and $0$
\begin{equation}
  \label{eq:Bdiv}
  \Ri F_*[F^*\omega_{\mc{Y}/S} \rightarrow \omega_{\mc{X}/S}]
\end{equation} 
in $\mc{D}^b(\mc{Y})$. One has to mention that though results of \cite{BranchFP} are stated for $\C$ they're valid in more generality. Indeed the fact \ref{eq:Bdiv} is a perfect torsion complex results from \cite[5 \& 7]{BranchFP} which do not depend on the characteristic (\cite[1.1]{BranchFP}). Following the proof of \cite[9]{BranchFP} what is needed is generic smoothness of $F$ on non-contracted components of closed fibers. Since $F$ is finite this is true as soon as $F$ is separable on each fiber. To make sure this is satisfied we either need $\bs{p}=0$ or $\bs{p} \nmid k_i$ for every $i$ and both are guaranteed by \ref{tam}. We shall therefore use needed results of \cite{BranchFP} in this generality. Following \cite{BranchFP} we have a branch morphism 
\begin{displaymath}
  \begin{tikzpicture}[baseline=(current  bounding  box.center),>=latex]
    \matrix (m) 
    [matrix of math nodes, row sep=2.5em, column sep=2em, text height=1ex, text depth=0.25ex] 
    {  \Br : \fMm_{g, n}(\bP^1, \vk) & \Div^{2g-2+2d}\bP^1\\};  
    \path[->,font=\scriptsize]  
    (m-1-1) edge (m-1-2);
  \end{tikzpicture}
\end{displaymath}
 that sends a stable map in $\fMm_{g,n}(\bP^1, \vk)$ on an effective Cartier divisor on $\bP^1$. It does coincide on the locus of $\fMm_{g,n}(\bP^1, \vk)$ given by smooth underlying curves with the classical branch morphism.   
\begin{lemma}
  \label{lem:aciter}
  The restriction of the branch divisor of a map $F : X \rightarrow \bP^1$ in $\fMm_{g, n}(\bP^1, \vk)$ to $\bP^1\setminus \infty$ is an effective divisor of degree $r=2g-2+d+n$.    
\end{lemma}
\begin{proof}
  Let $F_v$ be the restriction of $F$ to the connected component $X_v$ in the normalization of $X$ and write $\mf{n}$ for the Weil divisor of nodal points in $X$. Following \cite[10 \& 11]{BranchFP} 
  \begin{displaymath}
    \Br(F) = 2F_{*}\mf{n} + \sum_{\{v\, \mid\, F_v \,\text{is constant}\}} \Br(F_{v})
    + \sum_{\{v \, \mid\, F_v\,\text{dominates $\bP^1$}\}} \Br(F_{v}).  
\end{displaymath}
The sum of the first two terms is an effective divisor. Now $F_v$ dominates $\bP^1$ if and only if $X_v$ contains a marked point. The last term thus contains$(d-n)\infty$ and this is the only contribution of $\infty$ to $\Br(F)$. We get that $\Br(F)-(d-n)\infty$ is effective of expected degree. 
\end{proof}
\begin{theorem}
  \label{thm:aciter}
  Let $F : \mc{X} \rightarrow \bP_S^1$ be an $S$-map in $\fMm_{g, n}(\bP^1, \vk)$ we have that 
  \begin{equation}
    \Br(F) = \Br(F)_{\fZ} + (d-n)\infty
  \end{equation}
  where $\Br(F)_{\fZ}$ is an effective relative divisor disjoint of $\infty$ along the fibers and of degree $r$.
\end{theorem}
\begin{proof}  
  Following lemma \ref{lem:aciter} it is enough to show $\Br(F)$ contains $(d-n)\infty$, i.e. that there exists a decomposition $\Br(F) = W + (n-d)\infty$ for an effective relative divisor $W$ away from $\infty$. Let $\nabla$ be the divisor $\nabla = \sum_{i=1}^n (k_i-1)D_i$. The morphism $\mr{d}^\omega F : F^*\omega_{\bP^1_S} \rightarrow \omega_{\mc{X}/S}$ is equal at the neighborhood of a point in the support of $\nabla$ to the classical $\mr{d}F : F^*\Omega_{\bP^1_S} \rightarrow \Omega_{\mc{X}/S}$. In particular $\mr{d}^\omega F$ factors through $\Omega_{\mc{X}/S}(-\nabla) \rightarrow \Omega_{\mc{X}/S}$. We get the distinguished triangle 
  \begin{displaymath}
    \begin{tikzpicture}[baseline=(current  bounding  box.center),>=latex]
      \matrix (m) 
      [matrix of math nodes, row sep=2.5em, column sep=2em, text
      height=1ex, text depth=0.25ex]  
      {F^*\omega_{\bP^1_S} & F^*\omega_{\bP^1_S} & 0\\ 
       \omega_{\mc{X}/S}(-\nabla) & \omega_{\mc{X}/S} &
       \omega_{\mc{X}/S} \otimes \sO_{\nabla}.\\}; 
      \path[->,font=\scriptsize]  
      (m-1-2) edge (m-1-3)
      (m-2-1) edge (m-2-2)
      (m-2-2) edge (m-2-3)
      (m-1-1) edge node[left] {$\mr{d}^\omega F$} (m-2-1)
      (m-1-2) edge node[auto] {$\mr{d}^\omega F$} (m-2-2)
      (m-1-3) edge (m-2-3);
      \draw[double, double distance=1.5 pt, font=\scriptsize]
      (m-1-1) -- (m-1-2);
    \end{tikzpicture}
  \end{displaymath}
  The direct images of each vertical complex in degrees $-1$ and $0$ by $\Ri F_*$define perfect torsion complexes (\cite[6, 7 \& 9]{BranchFP}). By standard properties of the $\Div$ operator we get that 
  \begin{align*}
    \Br(F) & = \Div \Ri F_*\Big[F^*\omega_{\bP^1_S} \rightarrow \omega_{\mc{X}/S}\Big]\\
& = \underbrace{\Div \Ri F_*\Big[F^*\omega_{\bP^1_S} \rightarrow \omega_{\mc{X}/S}(-\nabla)\Big]}_{(\star\star)} + \underbrace{\Div\Ri F_*\Big[\omega_{\mc{X}/S}\otimes \sO_{\nabla}\Big]}_{(n-d)\infty}. 
  \end{align*}
Let's show $(\star\star)$ is an effective relative divisor having support away from $\infty$. At the neighborhood of $\nabla$ the sequence $F^*\omega_{\bP^1_S} \rightarrow \omega_{\mc{X}/S} \rightarrow \omega_{\mc{X}/S}\otimes \sO_{\nabla}$ is the exact sequence 
\begin{equation}
  \label{eq:suiteblabla}
  \begin{tikzpicture}[baseline=(current  bounding  box.center),>=latex]
    \matrix (m) 
    [matrix of math nodes, row sep=2.5em, column sep=2em, text
    height=1ex, text depth=0.25ex]  
    {  0 &  F^*\Omega_{\bP^1_S} & \Omega_{\mc{X}/S}  &
      \Omega_{\mc{X}/S}\otimes \sO_{\nabla} & 0. \\}; 
    \path[->,font=\scriptsize]  
    (m-1-1) edge (m-1-2)
    (m-1-2) edge node[above] {$\mr{d}F$} (m-1-3)
    (m-1-3) edge (m-1-4)
    (m-1-4) edge (m-1-5);
  \end{tikzpicture}
\end{equation}
In particular, in the neighborhood of $\nabla$ we have that $[F^*\omega_{\bP^1_S} \rightarrow \omega_{\mc{X}/S}]$ is quasi-isomorphic to $\omega_{\mc{X}/S}\otimes \sO_{\nabla}$. Now $F^*\omega_{\bP^1_S} \rightarrow \omega_{\mc{X}/S}(-\nabla)$ is $\mr{d}F$ restricted to its image. Therefore , in the neighborhood of $\nabla$, $[F^*\omega_{\bP^1_S} \rightarrow \omega_{\mc{X}/S}(-\nabla)]$ is acyclic and, since $[\omega_{\mc{X}/S} \otimes \sO_{\nabla}]=0$ away from $\nabla$ we have that $[F^*\omega_{\bP^1_S}\rightarrow \omega_{\mc{X}/S}(-\nabla)]=[F^*\omega_{\bP^1_S} \rightarrow \omega_{\mc{X}/S}]$. In conclusion $[F^*\omega_{\bP^1_S} \rightarrow \omega_{\mc{X}/S}]$ is equal to the restriction of $[F^*\omega_{\bP^1_S} \rightarrow \omega_{\mc{X}/S}]$ to $U = \bP^1_S\setminus \infty$. According to proposition \cite[1 v.]{BranchFP} the branch divisor of this restriction complex is is the restriction along $U\rightarrow \bP^1$ of the branch divisor of $F$. In particular $(\star\star)$ is effective.
\end{proof}

Using the branch divisor we build a section of the $\bG_a$-torsor $\fMm_{g, n}(\bP^1, \vk) \rightarrow \fZ_{g, \vk}$. This way one can define the branch divisor of a polar part $(\pi : \mc{X} \rightarrow S, \underline{\varrho}) \in \fZ_{g, \vk}$. We focus first on the local case and assume that $S$ is the spectrum of a local $\bbk$-algebra $\bs{A}$ of maximal ideal $\mf{m}$. Let $\zeta$ be a lift of $\underline{\varrho}$ to a global section of $\sO_{\mc{X}}(\sum_ik_iD_i)$ and write $F_{\zeta}$ the map to $\bP^1_S$ in $\fMm_{g, n}(\bP^1, \vk)$ defined by $\{\zeta, \bs{1}\}$. 
\begin{lemma}
  \label{lem:algpoly}
Assume $\bs{A}$ is a local henselian ring and let $D$ be an effective relative divisor of $\bA_{\bs{A}}^1$ defined by $\bs{I}\subset \bs{A}[t]$ and having constant degree $d$ along the fibers. Then $\bs{I} = (P(t))$ for a unitary polynomial in $\bs{A}[t]$. In particular $D$ defines a relative divisor of $\bP^1$ supported in $\bP^1\setminus \infty$.   
\end{lemma}
\begin{proof}
Let $\bs{B}$ be the algebra $\bs{A}[t]/\bs{I}$ defined by $D$. Since $D$ is a relative divisor over $\Spec(\bs{A})$ the morphism $\Spec(\bs{B}) \rightarrow \Spec(\bs{A})$ is quasi-finite and of finite type, in particular quasi-finite at each point $[\eta]$ of $\Spec(\bs{B})$ over $[\mf{m}]$. Following \cite[3]{RaynaudAH} page 76 we have $\bs{B} = \prod_{j=1}^m \bs{B}_{[\eta_j]} \times \bs{C}$, where $[\eta_j]$ are the points over $[\mf{m}]$ and $\bs{C}$ an $\bs{A}$-algebra whose restriction to the central fiber is $0$. Each $\bs{B}_{[\eta_i]}$ is finite and thus free of rank $d_i$ over $\bs{A}$. The assumption on the degree of $D$ along the fibers implies $\bs{C} = 0$ and consequently $\bs{B}$ is finite and free of rank $d$ over $\bs{A}$. 

Since $\bs{B}$ is flat over $\bs{A}$, $\bs{I}\otimes \bbk$ is a (principal) ideal of $\bbk [t]$. Let $P(t) \in \bs{I}$ be a lift of a generator of $\bs{I}\otimes \bbk$. We claim that $\bs{I} = (P(t))$. Each maximal ideal $\bs{M}$ of $\bs{A}[t]$ is an ideal over $\mf{m}$. Localize $\bs{A}[t]$ at $\bs{M}$. By assumption the restriction of the multiplicative endomorphism defined by $P$ to the central fiber is injective. Following the local criterium of flatness (\cite[(10.2.4)]{EGA31}) and since $\bs{A}[t]_{\bs{M}}$ is flat over $\bs{A}$, we get that $P$ still defines an injective endomorphism on $\bs{A}[t]_{\bs{M}}$  (i.e. $P$ is not a zero divisor) and that $\bs{A}[t]_{\bs{M}}/(P)$ is flat on $\bs{A}$.  The exact sequence
\begin{displaymath}
\begin{tikzpicture}[baseline=(current  bounding  box.center),>=latex]
    \matrix (m) 
    [matrix of math nodes, row sep=2.5em, column sep=2em, text height=1ex, text depth=0.25ex] 
    {  0 &  \bs{I}/(P) & \bs{A}[t]/(P) & \bs{A}[t]/\bs{I} & 0. \\}; 
    \path[->,font=\scriptsize]  
    (m-1-1) edge (m-1-2)
    (m-1-2) edge (m-1-3)
    (m-1-3) edge (m-1-4)
    (m-1-4) edge (m-1-5);
  \end{tikzpicture}  
\end{displaymath}
says $(\bs{I}/(P))_{\bs{M}} \otimes \bbk = 0$ for any $\bs{M}$. By Nakayama's lemma $\bs{I}_{\bs{M}} = (P)_{\bs{M}}$ and thus $\bs{I} = (P)$.
\end{proof}
Following theorem \ref{thm:aciter} the divisor $\Br_{\fZ}(F_\zeta)$ has constant degree along the fibers. Lemma \ref{lem:algpoly} implies the ideal of $\bA^1_{\bs{A}} = \bP^1\setminus \infty$ defined by $\Br_{\fZ}(\Phi_\zeta)$ is generated by $P(t) = b_r t^r + b_{r-1}t^{r-1} + \cdots + b_0$ with $b_r \in \bs{A}^*$.
\begin{definition}
\label{defn:normalized}
For general $S \in \Sch/\bbk$ we say $F : \mc{X} \rightarrow \bP^1_S$ in $\fMm_{g, n}(\bP^1, \vk)$ is normalized if locally on the base there is a local equation of $\Br_{\fZ}(F)$ has zero $r-1$ coefficient. If $F$ is of the form $F_{\zeta}$ for a global section $\zeta$ of $\sO_{\mc{X}}(\sum_ik_iD_i)$ we say $\zeta$ is normalized if $F_{\zeta}$ is.
\end{definition}
We shall further assume $\bs{A}$ is reduced. In this case the fact a local equation of $\Br(F_\zeta)$ has zero $r-1$ coefficient does not depend on the local equation. Indeed if $P(t)$ is such a local equation any other equation is of the form $\lambda P(t)$ for $\lambda \in \bs{A}[t]^* =\bs{A}^*$.
\begin{proposition}
  \label{prop:recollement}
  Let $(\pi : \mc{X} \rightarrow S, \underline{\varrho})$ be a polar curve in $\fZ_{g, \vk}$ over a local reduced scheme. There is a unique normalized lift of $\underline{\varrho}$ to $\Hi^0(\sO_{\mc{X}}(\sum_{i=1}^nk_iD_i))$. 
\end{proposition}
\begin{proof}
  Write $S = \Spec(\bs{A})$ and let $\zeta$ be a lift of $\underline{\varrho}$. Denote $\varepsilon_{a}$ translation by $a$ on $(\bP_{\bs{A}}^1, \infty)$. Since $\varepsilon_a$ is an automorphism we get that   
  \begin{displaymath}
    \Br_{\fZ}(\varepsilon_a\circ F_{\zeta}) = \Div\varepsilon_{a *}\Ri F_{\zeta*}[(F_{\zeta})^*\omega_{\bP_{S}^1/S} \rightarrow \omega_{\mc{X}/S}] = \varepsilon_{a *}\Div\Ri F_{\zeta*}[(F_{\zeta})^*\omega_{\bP_{S}^1/S} \rightarrow \omega_{\mc{X}/S}] 
\end{displaymath}
and this is $\varepsilon_{a *}\Br_{\fZ}(F_{\zeta})$. By lemma \ref{lem:algpoly} $\Br_{\fZ}(F_{\zeta})$ is generated by a unitary polynomial $P(t) = \sum_{0 \leq \ell \leq r} b_\ell t^\ell$ of degree $r$. The polynomial $P(t+a)$ is a generator of $\Br_{\fZ}(F_{\zeta+ a\bs{1}})$ whose $r-1$ coefficient is $b_{r-1} + r b_r a$. By assumption \ref{tam} $r$ is not zero in $\bbk$ and $b_r \in \bs{A}^*$, thus there is a unique $a$ for which $P(t)$ has zero $r-1$ coefficient given by $a= -(r b_r)^{-1}b_{r-1}$. Any other equation of $\Br(F_\zeta)$ is of the form $\lambda P(t)$ with $\lambda \in \bs{A}^*$ and $a= -(r b_r)^{-1}b_{r-1}$ works as well.   
\end{proof}
\begin{corollary}
  \label{cor:previous}
  Let $T$ and $S$ be reduced schemes and $(\pi : \mc{X} \rightarrow S, \underline{\varrho})$ be in $\fZ_{g, \vk}$. There is a unique normalized lift $\zeta_{\underline{\varrho}}$ of $\underline{\varrho}$ and for $f : T \rightarrow S$ we have $f^*\zeta_{\underline{\varrho}} = \zeta_{f^*\underline{\varrho}}$. 
\end{corollary}
\begin{proof} 
One can locally uniquely lift $\underline{\varrho}$ into a normalized section $\zeta_{\underline{\varrho}}$. This is enough to show lifts glue globally and uniquely over $S$. Now $f^*\zeta_{\underline{\varrho}}$ is a lift of $f^*\underline{\varrho}$, we only need to check it is normalized. The branch divisor commutes to base change : If $P(t)$ is a local equation for $\Br(F_{\zeta_\varrho})$ having zero $r-1$ coefficient then so is the case of $P(t) \otimes_{\bs{A}} 1_{\bs{B}}$ for $f^*\zeta_{\underline{\varrho}}$.
\end{proof}
\begin{corollary}
  \label{cor:immersionfermerZ}
  The stack $\fZ_{g, \vk}$ is naturally embedded in $\fMm_{g, n}(\bP^1, \vk)$.
\end{corollary}
\begin{proof}
  It is enough to build this embedding on closed points of $\fZ_{g, \vk}$. For this to make sense one has to notice $\fZ_{g, \vk}$ is reduced and $\fMm_{g, n}(\bP^1, \vk)$ is separated. The latter comes from the fact $\fMm_{g, n}(\bP^1, \vk)$ is a locally open substack in $\fMm_{g, n}(\bP^1, d, \infty)$ which is proper by \cite{BehManin}. For the former notice that an infinitesimal automorphism of a point of $\fZ_{g, \vk}$ is an infinitesimal automorphism of the underlying polar curve. The proof of proposition \ref{prop:tame} shows there is no such automorphisms assuming \ref{tam}. Each point $x = (\pi : X \rightarrow \mbb{L} , \underline{\varrho})$ in $\fZ_{g, \vk}$ defines a point $F_{\zeta_{\underline{\varrho}}}$ in $\fMm_{g,n}(\bP^1, \vk)$. Corollary \ref{cor:previous} shows this map does not depend on the equivalence class of $x$. It is clear this is a section of $\fMm_{g, n}(\bP^1, \vk) \rightarrow \fZ_{g, \vk}$ on points. This is as well the case on $\fZ_{g, \vk}$ for it is reduced and separated.  
\end{proof}
\begin{definition}
  The composition of the previous embedding with the branch morphism $\Br_\fZ$ is the extension of the $\mc{L}\!\mc{L}$ map to $\fZ_{g, \vk}$. It sends an object of $\fZ_{g, \vk}$ on an effective relative divisor of degree $2g-2+d+n$ in $\bP^1\setminus \infty$. It is denoted $\Br_{\fZ}$ as well. 
\end{definition}
The ELSV cone $\fZ_{g, \vk}$ is enclosed in a diagram 
\begin{displaymath}
  \begin{tikzpicture}[baseline=(current  bounding  box.center),>=latex] 
      \matrix (m) 
      [matrix of math nodes, row sep=2.5em, column sep=2.5em, text
      height=1.5ex, text depth=0.25ex]  
      { \fMm_{g, n} & \fZ_{g, \vk} & \Div^r\bA^1. \\};    
      \path[->,font=\scriptsize]  
      (m-1-2) edge node[auto] {$\Br_{\fZ}$} (m-1-3)
      (m-1-2) edge (m-1-1);
  \end{tikzpicture}
\end{displaymath} 
where the left hand-side morphism is the forgetful one. This is the expected behavior of a stack of covers. Recall the theorem on symmetric functions identifies $\Div^r\bA^1$ and $\bA^r$. An image of an object in $\fZ_{g, \vk}$ by $\Br_{\fZ}$ has zero first coordinate in $\bA^r$, this is the coordinate corresponding to the first elementary symmetric function. Assume for simplicity $x_1=0$. The set theoretic image of $\fZ_{g, \vk}$ by $\Br_{\fZ}$ is contained in the hyperplane $x_{1}=0$, i.e. an affine space of dimension $r-1$. For degree reasons this image is not zero. The $\bG_m$-action on $\fZ_{g, \vk}$ composing with automorphisms of $\Aut(\bP^1, 0, \infty)$ corresponds to the $\bG_m$-action on $\bA^r=\Div^r\bP^1$ having weights $(1, 2, \ldots, r)$. The induced action on the hyperplane $x_1=0$ has weights $(2, \ldots, r)$. Taking quotients we get the diagram 
\begin{equation}
  \label{corsp:Z}
  \begin{tikzpicture}[baseline=(current  bounding  box.center),>=latex] 
      \matrix (m) 
      [matrix of math nodes, row sep=2.5em, column sep=2.5em, text
      height=1.5ex, text depth=0.25ex]  
      { \fMm_{g, n} & \bP\fZ_{g, \vk} & \bP(2, \ldots, r) \\};    
      \path[->,font=\scriptsize]  
      (m-1-2) edge node[auto] {$\delta$} (m-1-3)
      (m-1-2) edge (m-1-1);
  \end{tikzpicture}
\end{equation} 
which is to be compared to \ref{corsp:H}. This is the needed configuration to pull out the ELSV formula. The $\delta$ morphism is finite of degree the Hurwitz number $\mu_{g, \vk}$ and an irreducible component of $\bP\fZ_{g, \vk}$ is of codimension at most $g$. The ELSV compactification $\cHh_{g, \vk}$ is exactly of codimension $g$ and thus of dimension $r-2$. Using the projection formula one gets
\begin{equation}
  \label{eq:intersection}
  \int_{\fMm_{g, n}} c_1\sO(1)^{r-2}\cap [\cHh_{g, \vk}] = \deg(\delta)\int_{\bP(2, \ldots, r)} c_1\sO(1)^{r-2} = \frac{\deg(\delta)}{r!}.
\end{equation}
Since $[\cHh_{g, \vk}]$ is a reduced component of $\bP\fZ_{g, \vk}$ it is of multiplicity $1$ in $[\bP\fZ_{g, \vk}]$. A simple computation shows that the union of components different of $\cHh_{g, \vk}$ is sent by $\delta$ on a closed subset of codimension $\geq 1$ in $\bP(2, \ldots, r)$. Summing these up we get that  
\begin{displaymath}
  \int_{\fMm_{g, n}} c_1\sO(1)^{r-2}\cap [\cHh_{g, \vk}] = \int_{\fMm_{g, n}}c_1\sO(1)^{r-2}\cap [\bP\fZ_{g, \vk}^g]
\end{displaymath} 
where $[\bP\fZ_{g, \vk}^g]$ is the part of $[\bP\fZ_{g, \vk}]$ supported on components of codimension $g$ in $\bP\fZ_{g, \vk}$. To show the contribution of components not in $\cHh_{g, \vk}$ is zero one can use the projection formula as in \ref{eq:intersection}. The contribution on the right-hand side is precisely the one defined by the section of $\mbb{E}_{g, \vk}$ defining $\fZ_{g, \vk}$. The precise computation is made in \cite[5]{ELSV}.

\section{Rigidifications and the Hurwitz projection}

\subsection{The stack of rigidified marked curves}

Let $(Y, q_\infty, \{q_j\}_{j=1}^{\bar{r}})$ be a prestable marked curve. Its underlying dual graph $\Gamma_Y$ is a tree. Following \cite[10.2]{OdaSesh} the collection of restrictions of a line bundle to each irreducible component of $Y$ defines an isomorphism   
\begin{displaymath}
  \Pic(Y) = \prod_{v \in V(\Gamma_Y)} \Pic(Y_v) \simeq \Z^{\Card{V(\Gamma_Y)}}.
\end{displaymath}
Since each irreducible component of $Y$ is a rational smooth curve a line bundle $\ms{M}$ on $Y$ is given its multi-degree, i.e. the collection of degrees on each irreducible component. The invertible sheaf of multi-degree $\{d_v\}_{v \in V(\Gamma_Y)}$ is written denoted $\sO_Y(\{d_v\})$. For instance the canonical sheaf on $Y$ has multi-degree $\{-2 + \val(v)\}_{v \in V(\Gamma_Y)}$.
\begin{lemma}
  \label{lem:h0h1}
  Given a collection $\{d_v\}_{v \in V(\Gamma_Y)}$ of \emph{natural} integers we have that 
  \begin{displaymath}
   h^1\big(\sO_Y(\{d_v\})\big) = 0 \quad \textrm{and} \quad h^0\big(\sO_Y(\{d_v\})\big) = 1 + \sum_{v \in V(\Gamma_v)} d_v.
   \end{displaymath}
\end{lemma}
\begin{proof}
  By Serre duality $h^1(\sO_Y(\{d_v\})) = h^0(\omega_Y(\{-d_v\}))$. If $Y$ is smooth $\omega_Y(\{-d_v\})$ has negative degree and $\Hi^0(\omega_Y(\{-d_v\})) = 0$. For general $Y$ let $u$ be a leaf of $\Gamma_Y$. The component $Y_{u}$ can have only one nodal point and the restriction of $\omega_Y(\{-d_v\})$ to $Y_{u}$ is still of negative degree. If $Y^*_{u}= \overline{Y\setminus Y_{u}}$ we have an isomorphism from $\Hi^0(\omega_Y(\{-d_v\}))$ on $\Hi^0(\omega_Y(\{-d_v\})_{\mid Y^*_{u}})$. Since $Y^*_{u}$ is a prestable genus $0$ curve having $\Card{V(\Gamma_Y)}-1$ components we get the first statement by induction. The second is a special case of \cite[10.4]{OdaSesh}.    
\end{proof}
Let $Y_\infty$ be the component of $Y$ containing $q_\infty$. We have an isomorphism $\Hi^0(\sO_Y(q_\infty))\simeq \Hi^0(\sO_{Y_\infty}(q_\infty))$. Since the later is isomorphic to $\sO_{\bP^1}(1)$ it is base point free. By lemma \ref{lem:h0h1} $h^0(\sO_{Y}(q_\infty)) = 2$ and 
\begin{displaymath}
\bigoplus_{\ell\geq 0}\Hi^0(\sO_Y(\ell q_\infty)) = \Sym{\Hi^0(\sO_Y(q_\infty))}.
\end{displaymath}
We thus get a map $Y \rightarrow \Proj{\Hi^0(\sO_Y(q_\infty))}$ which, after choosing a base of $\Hi^0(\sO_Y(q_\infty))$, gives a morphism $Y \rightarrow \bP^1$. It is an isomorphism if restricted to $Y_\infty$ and contracts any other component. This still makes sense in the relative case. Let $(\eta : \mc{Y} \rightarrow S, E_\infty, \{E_{j}\}_{j})$ be a genus $0$  prestable curve. By lemma \ref{lem:h0h1} $h^0(\sO_{\mc{Y}_s}(E_{\infty}))=2$ and $h^1(\sO_{\mc{Y}_s}(E_\infty)) = 0$ for each point $s \in S$. Following \cite[lemme 1.5]{Knudsen2} the direct image $\ms{Q}_\infty = \eta_*\sO_{\mc{Y}}(E_\infty)$ is a locally free sheaf of rank $2$ that commutes to base change. It is relatively ample and thus gives an $S$-map $\mc{Y} \rightarrow \bP_S(\ms{Q}_\infty)$. 
  
Let $\bs{1}$ be the section $\sO_S \rightarrow \ms{Q}_\infty$ obtained by pushing forward the canonical section $\sO_{\mc{Y}} \rightarrow \sO_{\mc{Y}}(E_\infty)$. If there is a section $\xi \in \Hi^0(\ms{Q}_\infty)$ such that $\{\bs{1}, \xi\}$ locally generate $\ms{Q}_\infty$, or equivalently trivialize $\ms{Q}_\infty$, then $\{\bs{1}, \xi\}$ defines an isomorphism $\bP(\ms{Q}_\infty) \simeq \bP^1_S$ and by composition a morphism $\mc{Y} \rightarrow \bP^1_S$. Call a \emph{rigidification} of $(\eta : \mc{Y} \rightarrow S, E_\infty, \{E_j\}_j)$ such a choice of $\xi$. A rigidication of $(\eta : \mc{Y} \rightarrow S, E_\infty, \{E_j\}_j)$ is equivalently given by a splitting of 
\begin{equation}
\label{exactseq}
\begin{tikzpicture}[baseline=(current  bounding  box.center),>=latex]
    \matrix (m) 
    [matrix of math nodes, row sep=2.5em, column sep=2em, text height=1ex, text depth=0.25ex] 
    {  0 &  \sO_S & \ms{Q}_\infty & \sL_\infty^{-1} & 0 \\}; 
    \path[->,font=\scriptsize]  
    (m-1-1) edge (m-1-2)
    (m-1-2) edge (m-1-3)
    (m-1-3) edge node[auto] {$\kappa$}(m-1-4)
    (m-1-4) edge (m-1-5);
  \end{tikzpicture}  
\end{equation}
together with a global nonzero section of $\sL_\infty^{-1}$. Given two rigidifications $\xi$ and $\xi'$ there is a unique $(\lambda, a) \in \Hi^0(\sO_S^*) \times \Hi^0(\sO_S)$ such that $\xi' = \lambda \xi + a \bs{1}$. This is an action of $\bG_m \ltimes \bG_a$ for which the set of rigidifications of $\eta$ (if not empty) is a torsor. This action corresponds to the one (by composition) of $\Aut(\bP^1, \infty)$ on the map $\mc{Y} \rightarrow \bP_S^1$ defined by $\{\bs{1}, \xi\}$.

Let $\psi : \mc{Y} \rightarrow \bP_S(\ms{Q}_\infty)$ be the map defined by $\ms{Q}_\infty$. The section $\tau_{j}$ and $\tau_\infty$ defining $E_j$ and $E_\infty$ give rise to the $S$-sections of $\bP_S(\ms{Q}_\infty)$ given by composing with $\psi$. We get marked $S$-points $\bar{E}_j$ which do not intersect $\bar{E}_\infty$. Each marked point $\bar{E}_j$ is defined by a onto map 
\begin{displaymath}
\begin{tikzpicture}[baseline=(current  bounding  box.center),>=latex]
    \matrix (m) 
    [matrix of math nodes, row sep=2.5em, column sep=2em, text height=1ex, text depth=0.25ex] 
    {\kappa_j : \ms{Q}_\infty  & \sL_j^{-1}\\};  
    \path[->>,font=\scriptsize]  
    (m-1-1) edge (m-1-2);
  \end{tikzpicture} 
\end{displaymath}
 where $\sL_j^{-1}$ is the conormal bundle at $\bar{E}_j$. The marked point $\bar{E}_\infty$ is given by $\kappa$ in the exact sequence \ref{exactseq}. The fact $\bar{E}_j$ does not intersect $\bar{E}_\infty$ corresponds to 
\begin{displaymath}
\begin{tikzpicture}[baseline=(current  bounding  box.center),>=latex]
    \matrix (m) 
    [matrix of math nodes, row sep=2.5em, column sep=2em, text height=1ex, text depth=0.25ex] 
    {\sO_S & \ms{Q}_\infty & \sL_{j}^{-1}\\};  
    \path[->,font=\scriptsize]  
    (m-1-1) edge (m-1-2)
    (m-1-2) edge (m-1-3);
  \end{tikzpicture} 
\end{displaymath} 
being an isomorphism. Given a rigidification $\xi$ of $\eta$ we get elements $\kappa_{j} (\xi) \in \Hi^0(\sL_{j}^{-1})$. In terms of the base of $\ms{Q}_\infty$ defined by $\{\bs{1}, \xi\}$ and using $\sL_j^{-1} \simeq \sO_S$ the morphism $\kappa_j$ is defined by $(\lambda_j, \mu_j) \in \Hi^0(\sO_S^*) \times \Hi^0(\sO_S)$. In this case $\bar{E}_j$ is the $S$-point of $\bP^1_S$ defined by $\mu_j\lambda_j^{-1}$.   
\begin{definition}
  Let $\bs{\epsilon} = \{\epsilon_j\}_{j=1}^{\bar{r}}$ be a set of partitions of $d$. An $\bs{\epsilon}$-rigidification of $(\eta :\mc{Y} \rightarrow S, \{E_j\}, E_\infty)$ is a rigidification $\xi$ such that 
  \begin{equation}
    \label{eq:etarig}
    \sum_{j} \big(\sum_{\ell_j}(\epsilon_{j}(\ell_j)
    -1)\big)\kappa_j (\xi) \in \mathrm{Nil}(\sO_S).
  \end{equation}
\end{definition}
One can see the set of $\bs{\epsilon}$-rigidifications of $\eta$ (if not empty) is a $\bG_m$-torsor. For reduced $S$ there can be only one $\bs{\epsilon}$-rigidification up to multiplication by $\lambda \in  H^0(\sO_S^*)$.

Let $\fMm_{0, \bar{r}, \infty}$ be the stack $\fMm_{0, \bar{r}+1}$ of curves having a marked point indexed by $\infty$. Write $\fMm_{0, \bar{r}, \infty}^{\bs{\epsilon}}$ for the algebraic stack whose objects are curves in $\fMm_{0, \bar{r}, \infty}$ endowed with an $\bs{\epsilon}$-rigidification. Let $\fMm_{0, \bullet, \infty}$ be the finite disjoint union of $\fMm_{0, \bar{r}, \infty}^{\bs{\epsilon}}$ for all partitions $\bs{\epsilon}$ of $d$.
\begin{lemma}
  The map $\fMm_{0, \bar{r},\infty}^{\bs{\epsilon}} \rightarrow \fMm_{0, \bar{r}, \infty}$ is a $\bG_m$-torsor and $\fMm_{0, \bar{r}, \infty} = [\fMm_{0, \bar{r},\infty}^{\bs{\epsilon}}/\bG_m]$.
\end{lemma} 
\begin{proof}
  To give a $\bG_m$-torsor over $\fMm_{0, \bar{r}, \infty}$ is to give a line bundle over $\fMm_{0,\bar{r}, \infty}$. Now an $\bs{\epsilon}$-rigidification of $(\eta : \mc{Y} \rightarrow S, \{E_j\}, E_\infty)$ is a global non zero section of $\sL_{\infty}^{-1}$. Equivalently it is a section of the $\bG_m$-torsor 
\begin{displaymath}
\Spec_S\big(\bigoplus_{\ell \in \Z} \sL_{\infty}^{\ell}\big).
\end{displaymath}
Thus $\fMm_{0, \bar{r}, \infty}^{\bs{\epsilon}}$ is the $\bG_m$-torsor defined by $\sL_\infty$ over $\fMm_{0, \bar{r}, \infty}$.   
\end{proof} 
The stack $\fMm_{0, \bar{r}, \infty}^{\bs{\epsilon}}$ comes with a natural map on the divisors of degree $r$ of $\bP^1$. It is related to a construction of M.~Kapranov in \cite[C]{Kapranover}. 
\begin{definition}
  \label{defn:morphkap}
  Denote by $\mf{k}_{\bs{\epsilon}}^\diamond$ (resp. $\mf{k}_{\bullet}^\diamond$) the morphism sending $((\eta:\mc{Y} \rightarrow S, \{E_j\},E_\infty), \xi)$ in $\fMm_{0, \bar{r}, \infty}^{\bs{\epsilon}}$ (resp. $\fMm_{0, \bullet, \infty}$) on the Cartier Divisor
  \begin{equation}
    \label{eq:kapdiv}
    \mf{c}[\xi]_*\big(\sum_{j, \ell_j}(\epsilon_j(\ell_j)-1)E_{j}\big)
  \end{equation}
  where $\mf{c}[\xi]$ is the contraction $\mc{Y} \rightarrow \bP^1_S$ defined by $\xi$.  
\end{definition}
Recall $\Div^\ell\bP^1= \bA^\ell$. The $\bG_m$-action on the rigidification $\xi$ corresponds to the action of $\Aut(\bP^1, 0, \infty)$ on $\mf{c}[\xi]$ by composition. It acts with weights $(1, 2, \ldots, \ell)$ on $\bA^\ell$. Now \ref{eq:kapdiv} is a degree $\ell$ divisor whose first coordinate (at least for reduced $S$) in $\bA^\ell$ is zero. Since all the stacks involved are reduced and separated we get that $\mf{k}_{\bs{\epsilon}}^\diamond$ factors through the hyperplane of zero first coordinate. Since $\ell$ cannot be zero taking the quotient by the $\bG_m$ actions on the source and target we get maps $\mf{k}_{\bs{\epsilon}}$ and $\mf{k}_\bullet$ respectively from $\fMm_{0, \bar{r}, \infty}$ and $\fMm_{0, \bullet, \infty}$ on $\bP(2, \ldots, \ell)$.

\subsection{The Hurwitz projection}
Let $\bs{\epsilon} = \{\epsilon_j\}_{j=1}^{\bar{r}}$ be a set of partitions of $d$. Denote $\fHh_{g, \vk, \diamond}^{\bs{\epsilon}}$ the fiber product 
\begin{displaymath}
\fHh_{g, \vk, \diamond}^{\bs{\epsilon}} = \fHh_{g, \vk}^{\bs{\epsilon}}\times_{\fMm_{0, \bar{r}, \infty}} \fMm_{0, \bar{r}, \infty}^{\bs{\epsilon}}.
\end{displaymath}
It is a line bundle over $\fHh_{g, \vk}^{\bs{\epsilon}}$ and has a natural action of $\bG_m$ coming from the one on $\fMm_{0, \bar{r}, \infty}^{\bs{\epsilon}}$. The quotient of $\fHh_{g, \vk, \diamond}^{\bs{\epsilon}}$ by this action is $\fHh_{g, \vk}^{\bs{\epsilon}}$. An $S$-object in $\fHh_{g, \vk, \diamond}^{\bs{\epsilon}}$ is an admissible cover 
\begin{displaymath}
\begin{tikzpicture}[baseline=(current  bounding  box.center),>=latex]
    \matrix (m) 
    [matrix of math nodes, row sep=2.5em, column sep=2em, text height=1ex, text depth=0.25ex] 
    {\phi : (\pi : \mc{X} \rightarrow S, \underline{D}, \{\underline{\bs{D}}_j\}_{j=1}^{\bar{r}}) & (\eta : \mc{Y} \rightarrow S, E_\infty, \{E_j\}_{j=1}^{\bar{r}})\\};  
    \path[->,font=\scriptsize]  
    (m-1-1) edge (m-1-2);
  \end{tikzpicture} 
\end{displaymath}
together with an $\bs{\epsilon}$-rigidification $\xi$ of $(\eta : \mc{Y} \rightarrow S,E_\infty, \{E_j\}_{j=1}^{\bar{r}})$. Composing $\phi$ with the morphism to $(\bP^1_S, \infty)$ defined by $\xi$ one gets a map 
\begin{displaymath}
\xi[\phi] : (\pi : \mc{X} \rightarrow S, \underline{D}) \rightarrow (\bP^1_S, \infty)
\end{displaymath}
which might not be a stable map. Denote $\overline{\xi[\phi]} : \overline{\mc{X}} \rightarrow \bP^1_S$ the stabilization of $\xi[\phi]$. 
\begin{lemma}
  The map $\overline{\xi[\phi]}$ is an object of $\fMm_{g, n}(\bP^1, \vk)$ lying in $\fZ_{g, \vk}$. 
\end{lemma}
\begin{proof}
  Following theorem \ref{prop:Sec}
  \begin{displaymath}
    \Br_{\fZ}(\overline{\xi[\phi]}) = \sum_{j} \big(\sum_{\ell_j}(\epsilon_j(\ell_j)-1)\big)\bar{E}_{j} 
  \end{displaymath}   
  where $\bar{E}_j$ is the $S$-point of $\bP^1_S$ obtained by pushing forward $E_j$. Condition \ref{eq:etarig} says one can locally find equations of $\Br_{\fZ}(\overline{\xi[\phi]})$ having zero $r-1$ coefficient. We thus only need to show $\overline{\xi[\phi]}$ is in $\fMm_{g, n}(\bP^1, \vk)$. It is enough to check this on the closed fibers. Let $s \in S$. We have to show $\overline{\xi[\phi]}_s$ has ramification index $k_i$ at $D_{i, s}$. But such a point lies on a component of $\mc{X}_s$ not contracted by $\xi[\phi]_s$, i.e. dominating $(\bP^1, \infty)$. Thus ocal behavior of $\xi[\phi]_s$ and $\phi_s$ at $D_{i, s}$ are the same and $\overline{\xi[\phi]}_s$ has ramification index $k_i$ at $D_{i, s}$. 
\end{proof}
The previous procedure defines a morphism $\mf{h}_{\bs{\epsilon}}^\diamond : \fHh_{g, \vk, \diamond}^{\bs{\epsilon}} \rightarrow \fZ_{g, \vk}$ that is equal to the forgetful map on the locus of underlying smooth curves. Write $\mf{h}^\diamond$ for the simply branched case away from $\infty$ and $\mf{h}_\bullet^\diamond$ for the disjoint union of $\mf{h}_{\bs{\epsilon}}^\diamond$ over all partitions of $d$. The morphisms $\mf{h}_{\bs{\epsilon}}^{\diamond}$ are $\bG_m$-equivariant for the action on $\bs{\epsilon}$-rigidificatoins and the one coming with $\fZ_{g, \vk}$. Because of stability of domain curves of admissible covers one can see $\mf{h}_{\bs{\epsilon}}^\diamond$ does not intersect the zero section of $\fP_{g, \vk}$. Let $\mf{h}_{\bs{\epsilon}}$ be the map $\fHh_{g, \vk} \rightarrow \bP\fZ_{g, \vk}$ obtained out of $\mf{h}_{\bs{\epsilon}}^{\diamond}$ by taking the quotients by the $\bG_m$-actions on domain and target.
\begin{definition}
  The Hurwitz projection $\mf{h}_\bullet$ is the disjoint union of $\mf{h}_{\bs{\epsilon}}$ over partitions of $d$. 
\end{definition}
\begin{theorem}
  \label{thm:finale}
  The Hurwitz projection $\mf{h}_\bullet$ factors through $\cHh_{g, \vk, \bullet}$ giving rise to a proper surjective morphism and a commutative diagram
  \begin{displaymath}
    \begin{tikzpicture}[baseline=(current  bounding  box.center),>=latex] 
      \matrix (m) 
      [matrix of math nodes, row sep=2.5em, column sep=2.5em, text
      height=1.5ex, text depth=0.25ex]  
      { \fMm_{g, n} & \fHh_{g, \vk}^\bullet & \fMm_{0, \bullet, \infty} \\
        \fMm_{g, n} & \cHh_{g, \vk}^\bullet & \bP(2, \ldots, r) \\};    
      \path[->,font=\scriptsize]  
      (m-2-2) edge (m-2-3)
      (m-2-2) edge (m-2-1)
      (m-1-2) edge node[auto] {$\mf{h}_\bullet$} (m-2-2)
      (m-1-3) edge node[auto] {$\mf{k}_\bullet$} (m-2-3)
      (m-1-2) edge (m-1-3)
      (m-1-2) edge (m-1-1);
      \draw[double, double distance=1.5 pt, font=\scriptsize]
      (m-1-1) -- (m-2-1);
  \end{tikzpicture}
  \end{displaymath}
\end{theorem}
\begin{proof}
  The fact the right-hand square is commutative is a direct consequence of theorem \ref{prop:Sec} and does not depend on the fact $\mf{h}_\bullet$ factors through $\cHh_{g,\vk}^\bullet$. Let's focus on the first statement. Let $|\mf{h}_\bullet| : |\fHh_{g,\vk}| \rightarrow |\bP\fZ_{g,\vk}|$ be the continuous map induced by $\mf{h}_\bullet$ at the level of points. It is clear $\mf{h}_\bullet$ is proper. This says $|\mf{h}_\bullet|$ is closed and thus each point in the closure of $|\mf{h}_\bullet||\fH_{g, \vk}^\bullet|$ is in the image by $\mf{h}_\bullet$ of a point in the closure of $|\fH_{g, \vk}^\bullet|$, i.e.  
\begin{displaymath}
  |\mf{h}_\bullet|\left|\fHh_{g, \vk}^\bullet\right| \supset
  \overline{|\mf{h}_\bullet|\left|\fH_{g, \vk}^\bullet\right|}. 
\end{displaymath}
Since $|\mf{h}_\bullet||\fH_{g, \vk}^\bullet|= |\cH_{g, \vk}^\bullet|$ one can as well write 
\begin{displaymath}
  |\mf{h}_\bullet|\left|\fHh_{g, \vk}^\bullet\right| \supset \overline{\left|\cH_{g,
  \vk}^\bullet\right|} = \left|\cHh_{g,\vk}^\bullet\right|.
\end{displaymath}
The inverse inclusion comes from the fact $\fH_{g, \vk}^\bullet$ is dense in $\fHh_{g,\vk}^\bullet$. The image of a point in the closure of $|\fH_{g, \vk}^\bullet|$ is a point in the closure of $|\mf{h}||\fH_{g,\vk}^\bullet| = |\cH_{g, \vk}^\bullet|$, i.e. a point in $|\cHh_{g, \vk}^\bullet|$. Now $\cHh_{g, \vk}^\bullet$ is a reduced separated stack and so is the case of $\fHh_{g, \vk}^\bullet$, the continuous map $|\mf{h}_\bullet| : |\fHh_{g,\vk}^\bullet| \rightarrow |\cHh_{g, \vk}^\bullet|$ does therefore extend into a morphism of stacks $\mf{h}_\bullet : \fHh_{g, \vk}^\bullet \rightarrow \cHh_{g, \vk}^\bullet$. The fact $\mf{h}_\bullet$ is onto comes from the equality $|\mf{h}_\bullet||\fHh_{g, \vk}^\bullet| = |\cHh_{g,\vk}^\bullet|$. 
\end{proof}
\begin{corollary}
  The restriction $\mf{h}$ of $\mf{h}_\bullet$ to $\fHh_{g, \vk}$ is proper and onto. In particular, $\cHh_{g, \vk} = \cHh_{g, \vk}^\bullet$. 
\end{corollary}
\begin{proof}
  Let $\phi : (X, \underline{p}, \{p_{j,\ell_j}\}) \rightarrow (Y, q_\infty, q_j) \in \fHh_{g, \vk}^{\bs{\epsilon}}$ be a point in $\fHh_{g,\vk}^\bullet$ whose image by $\mf{h}_\bullet$ is a point $F \in \cHh_{g, \vk}^\bullet$. Let $q_j$ be a non-simple branch point of $\phi$. We glue a projective line $R_{q_j}$ to $Y$ along $0$ and $q_j$ and projective lines $R_{p_{j, \ell_j}}$ to $X$ along $0$ and $p_{j, \ell_j}$ for fixed $j$. For each $\ell_j$ choose a ramified cover $R_{p_j, \ell_j} \rightarrow R_{q_j}$ simply branched away from $p_{j, \ell_j}$ where it has ramification index $\epsilon_j(\ell_j)$. We get this way a cover from $X\sqcup_{p_{j, \ell_j}} R_{p_{j, \ell_j}}$ on $Y\sqcup_{q_j}R_{q_j}$ whose image by $\mf{h}_\bullet$ is clearly $F$ and whose number of non-simple branched points is less than the one of $F$. By induction we get the desired simply ramified cover.
\end{proof}

\appendix

\section{Compatibility of branch divisors}

\begin{theorem}
  \label{prop:Sec}
  Let $\phi : \mc{X} \rightarrow \mc{Y}$ an admissible cover in $\fHh_{g, \vk}^\bullet$ over $S$ and $F : \overline{\mc{X}} \rightarrow \overline{\mc{Y}}$ an $S$-map from a prestable $S$-curve on a smooth genus $0$ $S$-curve. If we have a commutative diagram  
  \begin{equation}  
    \label{diag:depart}
    \begin{tikzpicture}[baseline=(current  bounding  box.center),>=latex] 
      \matrix (m) 
      [matrix of math nodes, row sep=3.5em, column sep=4em, text
      height=1.5ex, text depth=0.25ex]  
      {  \mc{X} &  \mc{Y} \\  
        \overline{\mc{X}} & \overline{\mc{Y}} \\};  
      \path[->,font=\scriptsize] 
      (m-1-1) edge node[auto] {$\phi$} (m-1-2) 
      edge node[left] {$\mf{c}_<$} (m-2-1)
      edge node[auto] {$\varphi$} (m-2-2)
      (m-1-2) edge node[auto] {$\mf{c}_>$} (m-2-2)         
      (m-2-1) edge node[below] {$F$}(m-2-2);
  \end{tikzpicture}
\end{equation}
for contractions $\mf{c}_{>}$ and $\mf{c}_{<}$ in the sense of \cite{Knudsen2} then $\Br(F) = \mf{c}_{>*}\Br(\phi)$
\end{theorem}
The section is dedicated to the proof of this theorem. We write down $\Br(\varphi)$ in terms of the data in the upper and lower triangles $\triangle_{up}$ and $\triangle_{lo}$. Comparing relations in both triangles gives the statement. We start with a general lemma. 
\begin{lemma}
  \label{lemcohann1}
  Let $\gamma : ( \pi : \mc{Z} \rightarrow S) \rightarrow (\eta : \mc{T} \rightarrow S)$ be a proper and onto $S$-mophism between prestable curves. Then for $\ell > 0$ we have $\Ri^\ell\gamma_*\omega_{\mc{Z}/S} = 0$.  
\end{lemma}
\begin{proof}
  Fibers of $\gamma$ have dimension $\leq 1$. By lemma \cite[1.4]{Knudsen2} $\Ri^{\ell}\gamma_*=0$ for $\ell > 1$. By the same lemma it is enough to show that for every closed point $t \in \mc{T}$ we have 
\begin{displaymath}
\Hi^1(\gamma^{-1}(t), \omega_{\mc{Z}/S \mid \gamma^{-1}(t)}) = 0
\end{displaymath}
to have $\Ri^1\gamma_*\omega_{\mc{Z}/S} = 0$. The fiber $\gamma^{-1}(t)$ is a sub-curve of $\mc{Z}_s$ for $s = \eta(t)$. Since $\gamma$ is onto $\gamma^{-1}(t) \neq \mc{Z}_s$. We are brought to show that given a sub-curve $X \subsetneq Z$ in a prestable curve $\Hi^1(\omega_{Z\mid X}) =  0$. We can assume $X \subsetneq Z$ is connected. Write $\mf{n}_{Z\mid X}$ the set of nodal points in $X\cap \overline{Z \setminus X}$. A point in $\mf{n}_{Z\mid X}$ is regular on $X$ and $\mf{n}_{Z\mid X} \neq \emptyset$. We have that $\omega_{Z \mid X} = \omega_{X}(\sum_{x \in \mf{n}_{Z\mid X}}x)$. Thus $\Hi^1(\omega_{Z \mid X}) = \Hi^1(\omega_{X}(\sum_{x \in  \mf{n}_{Z \mid X}}x))$ and the latter is zero by duality.
\end{proof}
\subsection*{Relations in $\triangle_{lo}$}
We have an isomorphism of complexes 
\begin{equation}
  \label{lem:triangleinf}
  \Ri\varphi_*[\varphi^*\omega_{\overline{\mc{Y}}/S} \rightarrow \omega_{\mc{X}/S}] \simeq \Ri F_*[F^*\omega_{\overline{\mc{Y}}/S} \rightarrow \omega_{\overline{\mc{X}}/S}] 
\end{equation}
In particular, we get the equality $\Br(\varphi) = \Br(F)$.
\begin{proof}
  By definition $\Br(\varphi) = \Div{\Ri \varphi_*[\varphi^*\omega_{\overline{\mc{Y}}/S} \rightarrow \omega_{\mc{X}/S}]}$ and thus 
  \begin{eqnarray}
    \label{eqnBrFP}
    \Br(\varphi) =\Div{\Ri F_*\circ \Ri\mf{c}_{<*}[\mf{c}_<^*F^*\omega_{\overline{\mc{Y}}/S} \rightarrow \omega_{\mc{X}/S}]}.  
  \end{eqnarray}
  We simplify the complex appearing on the right-hand side. Notice that $\mf{c}_{<}$ satisfies the assumptions of lemma \ref{lemcohann1} and thus $\Ri^\ell\mf{c}_{<*}\omega_{\mc{X}/S} = 0$ for every $\ell > 0$. We claim this is the case as well for $\Ri^\ell\mf{c}_{<*}\varphi^*\omega_{\overline{\mc{Y}}/S}$. Take a point $\bar{x} \in \overline{\mc{X}}$ over $s \in S$. Let $t = F(\bar{x}) \in \overline{\mc{Y}}$. We have that 
\begin{displaymath}
\varphi^*\omega_{\overline{\mc{Y}}/S \mid\mf{c}_<^{-1}(\bar{x})} = (\varphi^*\omega_{\overline{\mc{Y}}/S \mid \varphi^{-1}(t)})_{\mid\mf{c}_<^{-1}(\bar{x})}.
\end{displaymath}
By base change we get that 
\begin{displaymath}
\varphi^*\omega_{\overline{\mc{Y}}/S \mid \varphi^{-1}(t)} = \varphi_t^*\omega_{t} = \sO_{\varphi^{-1}(t)}
\end{displaymath}
and thus 
\begin{displaymath}
(\varphi^*\omega_{\overline{\mc{Y}}/S \mid \varphi^{-1}(t)})_{\mid \mf{c}_<^{-1}(\bar{x})} = \sO_{\mf{c}_<^{-1}(\bar{x})}.
\end{displaymath}
Finally using Serre duality we get that 
\begin{displaymath}
  \Hi^1(\mf{c}_>^{-1}(\bar{x}), \varphi^*\omega_{\overline{\mc{Y}}/S \mid \mf{c}_>^{-1}(\bar{x})}) = \Hi^0(\mf{c}_<^{-1}(\bar{x}), \omega_{\mf{c}_<^{-1}(\bar{x})})^{\vee}.   
\end{displaymath}
This last cohomology group is zero: the dual graph of the genus $0$ curve $\mf{c}_{<}^{-1}(\bar{x})$ is a tree. We prove our claim by induction on the number of irreducible components of $\mf{c}_{<}^{-1}(\bar{x})$. If $\mf{c}_{<}^{-1}(\bar{x})$ is smooth $\omega_{\mf{c}^{-1}(\bar{x})}$ has negative degree. Otherwise each leaf of $\mf{c}_{<}^{-1}(\bar{x})$ meets the rest of the curve in a unique nodal point. The restriction of $\omega_{\mf{c}_{<}^{-1}(\bar{x})}$ on each leaf has negative degree and any global section is zero on the leafs. We can reduce the number of components this way. By lemma \cite[1.4]{Knudsen2} $\Ri^\ell\mf{c}_{<*}\varphi^*\omega_{\overline{\mc{Y}}/S}=0$ for positive $\ell$. Now we have   
  \begin{displaymath}
    \Ri\mf{c}_{<*}[\mf{c}_<^*F^*\omega_{\overline{\mc{Y}}/S} \rightarrow \omega_{\mc{X}/S}] = [\mf{c}_{<*}\mf{c}_<^*F^*\omega_{\overline{\mc{Y}}/S} \rightarrow \mf{c}_{<*}\omega_{\mc{X}/S}].    
  \end{displaymath}
  Since $\mf{c}_<$ has connected fibers $\mf{c}_{<*}\sO_{\mc{X}} = \sO_{\overline{\mc{X}}}$ and the projection formula gives 
\begin{displaymath}
\mf{c}_{<*}\mf{c}_<^*F^*\omega_{\overline{\mc{Y}}/S} =
  F^*\omega_{\overline{\mc{Y}}/S}. 
\end{displaymath}
By \cite[3.11]{BehManin} or equivalently \cite[1.6]{Knudsen2} we have that $\mf{c}_{< *}\omega_{\mc{X}/S} = \omega_{\overline{\mc{X}}/S}$. Finally
  \begin{displaymath}
    \Ri\mf{c}_{<*}[\mf{c}_<^*F^*\omega_{\overline{\mc{Y}}/S} \rightarrow \omega_{\mc{X}/S}] = [F^*\omega_{\overline{\mc{Y}}/S} \rightarrow \omega_{\mc{X}/S}].  
  \end{displaymath}
  and the desired relation is obtained by comparing the previous equality with \ref{eqnBrFP}. 
\end{proof}

\subsection*{Relations in $\triangle_{up}$} The ramification divisor of $\phi$ is the $\Div$ of the complex in degrees $-1$ and $0$ given by $[\phi^*\omega_{\mc{Y}/S} \rightarrow \omega_{\mc{X}/S}]$. If $\mf{R}_\phi$ denotes the ramification divisor of $\phi$ and $\mf{B}_\phi$ its branch divisor we have that $\phi_*\sO_{\mf{R}_\phi} = \sO_{\mf{B}_\phi}$. In the case of an admissible such as $\phi$ we can extend the classical relation $\omega_{\mc{X}/S} =\sO(\mf{R}_\phi) \otimes \phi^*\omega_{\mc{Y}/S}$. For more details on the ramification divisor of an admissible cover see \cite{HurBR}. In $\triangle_{up}$ we have the relation 
\begin{equation}
  \label{lem:trianglesup}
  \Br(\varphi) = \Div{\mf{c}_{>*}\sO_{\mf{B}_\phi}}.
\end{equation}
In particular explicitly writing down the right-hand side gives $\Br(\varphi) = \mf{c}_{>*}\Div\sO_{\mf{B}_\phi}$. 
\begin{proof}
  It is rather straightforward to see the set of morphisms between vertical complexes  
  \begin{displaymath}
    \begin{tikzpicture}[baseline=(current  bounding  box.center),>=latex] 
      \matrix (m) 
      [matrix of math nodes, row sep=3em, column sep=2.5em, text
      height=1.5ex, text depth=0.25ex]  
      { \phi^*\mf{c}_>^*\omega_{\overline{\mc{Y}}/S}  &
        \varphi^*\omega_{\overline{\mc{Y}}/S} & \phi^*\omega_{\mc{Y}/S} \\   
        \phi^*\omega_{\mc{Y}/S}] & \omega_{\mc{X}/S} & \omega_{\mc{X}/S} \\};  
      \path[->,font=\scriptsize] 
      (m-1-1) edge node[left] {$\phi^*\mr{d}^\omega\mf{c}_>$} (m-2-1)
      (m-1-2) edge node[auto] {$\mr{d}^\omega\varphi$} (m-2-2)
              edge node[auto] {$\phi^*\mr{d}^\omega\mf{c}_>$} (m-1-3)
      (m-2-1) edge node[below] {$\mr{d}^\omega\phi$} (m-2-2)
      (m-1-3) edge node[auto] {$\mr{d}^\omega\phi$} (m-2-3);
      \draw[double, double distance=1.5 pt, font=\scriptsize]
      (m-2-2) -- (m-2-3)
      (m-1-1) -- (m-1-2);      
   \end{tikzpicture}
 \end{displaymath}
defines a distinguished triangle in $D^b(\mc{X})$. By lemma \cite[2]{BranchFP} we get that 
  \begin{displaymath}
    \Br(\varphi)   = \Div\underbrace{\Ri\varphi_*[\varphi^*\omega_{\overline{\mc{Y}}/S} \rightarrow \phi^*\omega_{\mc{Y}/S}]}_{(\star)} + \Div\underbrace{\Ri\varphi_*[\phi^*\omega_{\mc{Y}/S} \rightarrow \omega_{\mc{X}/S}]}_{(\star\star)} 
  \end{displaymath}
  Now 
  \begin{displaymath}
    (\star)  = \Ri\mf{c}_{>*}[\phi_*\phi^*\mf{c}_>^*\omega_{\overline{\mc{Y}}/S} \rightarrow \phi_*\phi^*\omega_{\mc{Y}/S}] = \Ri\mf{c}_{>*}[\phi_*\sO_{\mc{X}}\otimes \mf{c}_>^*\omega_{\overline{\mc{Y}}/S} \rightarrow \phi_*\sO_{\mc{X}}\otimes \omega_{\mc{Y}/S}]
  \end{displaymath}
  Since $\phi_*\sO_{\mc{X}}$ is locally free we get that $(\star)  = \Ri\mf{c}_{>*}(\phi_*\sO_{\mc{X}}\otimes [\mf{c}_>^*\omega_{\overline{\mc{Y}}/S} \rightarrow \omega_{\mc{Y}/S}])$. By proposition \cite[9]{Knudsen1} we have $\Div(\star) = \Div\Ri\mf{c}_{>*}[\mf{c}_>^*\omega_{\overline{\mc{Y}}/S} \rightarrow \omega_{\mc{Y}/S}]$. This is the branch divisor of $\mf{c}_>$. Restricted to the fibers it contracts trees of smooth rational components on the component marked by infinity. Following \cite[3 (22)]{BranchFP} one can show this divisor is zero. Getting back to $(\star\star)$ we have that
\begin{displaymath}
  (\star\star) = \Ri\varphi_*[\omega_{\mc{X}/S}\otimes \sO_{\mc{X}}(-\mf{R}_\phi) \rightarrow \omega_{\mc{X}/S}] = \Ri\varphi_*(\omega_{\mc{X}/S}\otimes [\sO_{\mc{X}}(-\mf{R}_\phi) \rightarrow \sO_{\mc{X}}]) = \Ri\varphi_*(\omega_{\mc{X}/S}\otimes \sO_{\mf{R}_\phi}). 
\end{displaymath}
Following \cite[9]{Knudsen1} we get that 
\begin{displaymath}
\Div(\star\star) = \Div\Ri\varphi_*\sO_{\mf{R}_\phi} = \Div \Ri\mf{c}_{>*}\sO_{\mf{B}_\phi}.
\end{displaymath}
Since $\sO_{\mf{B}_\phi}$ is supported along a scheme of relative dimension $0$ the higher derived functors $\Ri^\ell\mf{c}_{>*}\sO_{\mf{B}_\phi}$ are zero for $\ell > 0$. The desired divisor is thus equal to $\Div\mf{c}_{>*}\sO_{\mf{B}_\phi}$. 
\end{proof}

\end{document}